\documentclass[12pt]{amsart}
\usepackage{geometry}  
\usepackage{graphicx}
\usepackage{amssymb}
\usepackage{amsmath}
\usepackage{color}
\usepackage{pdfsync}
 \usepackage{tikz}
 \usepackage{tikz-cd}

\usepackage[all]{xy}
\xyoption{matrix}
\xyoption{arrow}

\usetikzlibrary{arrows,decorations.pathmorphing,backgrounds,positioning,fit,petri}

\newtheorem{thm}{Theorem}[section]
\newtheorem{lem}[thm]{Lemma}
\newtheorem{cor}[thm]{Corollary}
\newtheorem{prop}[thm]{Proposition}

 \newtheorem{innercustomthm}{Theorem}
\newenvironment{customthm}[1]
  {\renewcommand\theinnercustomthm{#1}\innercustomthm}{\endinnercustomthm}
  	
\newtheorem{innercustomcor}{Corollary}
\newenvironment{customcor}[1]
	{\renewcommand
	\theinnercustomcor{#1}\innercustomcor}
	{\endinnercustomcor}

\theoremstyle{definition}
\newtheorem{defn}[thm]{Definition}
\newtheorem{eg}[thm]{Example}

\theoremstyle{remark}
\newtheorem{rem}[thm]{Remark}
 
\numberwithin{equation}{section}

 \newcommand{\onto}{\twoheadrightarrow}

\DeclareMathOperator{\supp}{supp} 
\DeclareMathOperator{\Hom}{Hom}%
\DeclareMathOperator{\Ext}{Ext}%
\DeclareMathOperator{\undim}{\underline{dim}}
 
\newcommand{\field}[1]{\mathbb{#1}}
\newcommand{\ZZ}{\ensuremath{{\field{Z}}}}

\newcommand{\RR}{\ensuremath{{\field{R}}}}

\newcommand{\commentout}[1]{}

\newcommand{\cC}{\ensuremath{{\mathcal{C}}}}
\newcommand{\cD}{\ensuremath{{\mathcal{D}}}}

\newcommand{\no}[1]{}

\title{Exceptional sequences and rooted labeled forests}
\author{Kiyoshi Igusa}
\email{igusa@brandeis.edu}
\thanks{First author supported by Simons Foundation Grant \#686616.}

\author{Emre Sen}
\email{emresen641@gmail.com}



\begin{document}

\begin{abstract}
 We give a representation-theoretic bijection between rooted labeled forests with $n$ vertices and complete exceptional sequences for the quiver of type $A_n$ with straight orientation. The ascending and descending vertices in the forest correspond to relatively injective and relatively projective objects in the exceptional sequence. We conclude that every object in an exceptional sequence for linearly oriented $A_n$ is either relatively projective or relatively injective or both.
We construct a natural action of the extended braid group on rooted labeled forests and show that it agrees with the known action of the braid group on complete exceptional sequences. We also describe the action of $\Delta$, the Garside element of the braid group, on rooted labeled forests using representation theory and show how this relates to cluster theory.
\end{abstract}

\maketitle

\tableofcontents



\section*{Introduction}

Exceptional sequences of objects is a central subject in algebraic geometry, representation theory and combinatorics. For the Dynkin quiver of type $A_n$, there are $(n+1)^{n-1}$ complete exceptional sequences \cite{Seidel}. From the combinatorial point of view, there is a vast list of enumeration problems in which we encounter sets of size $(n+1)^{n-1}$ and indeed exceptional sequences of type $A_n$ can be interpreted as maximal chains in the poset (partially ordered set) of noncrossing partitions of the set $\{0,1,2,\cdots,n\}$, trees with $n$ labeled edges, factorizations of the cyclic permutation $(012\cdots n)$ as a product of $n$ transpositions \cite{GY}, \cite{BW}, \cite{GesselSeo}, \cite{Sokal}, \cite{Loo}, \cite{M}, \cite{Seidel}, \cite{IT}, \cite{RingelExcSeq}. Our aim is to give another combinatorial interpretation which has representation theoretical outcomes.

 An \emph{exceptional sequence} for a hereditary algebra $\Lambda$ over any field is a sequence of rigid indecomposable $\Lambda$-modules $(E_1,\cdots,E_k)$  so that 
\[
\Hom_\Lambda(E_j,E_i)=\Ext_\Lambda(E_j,E_i)=0
\]
for all $1\leq i<j\leq k$. The sequence is called \emph{complete} if $k$ is equal to the number of simple $\Lambda$-modules which is the number of vertices in the quiver of $\Lambda$. There is lot of work on exceptional sequences and their generalizations from the representation theory point of view: \cite{Seidel}, \cite{Ringel2}, \cite{BM}, \cite{BRT}, \cite{GIMO}, \cite{Ig-Sh}, \cite{MC}, \cite{MT}, \cite{Sen}, \cite{Sen2}. In this work, we consider a homological property of modules: which elements of an exceptional sequence can be either relatively projective or relatively injective. We recall that a term $E_j$ in an exceptional sequence is called \emph{relatively projective} if it is a projective object of the \emph{perpendicular category} which is the full subcategory of $mod$-$\Lambda$ of objects $X$ so that $\Hom(E_k,X)=0=\Ext(E_k,X)$ for all $k>j$. (See Definition \ref{def: relatively projective and perp cat}.) In particular, $E_1,E_2,\cdots,E_j$ lie in this perpendicular category. Relatively injective objects are defined dually. (See section \ref{ss: rel proj and rel inj}.) In the case of \emph{signed exceptional sequences} introduced in \cite{IT13}, relatively projective terms are allowed to be ``shifted''. This additional structure converts the poset of noncrossing partitions into a category \cite{I:ncp} and maximal chains in the poset become maximal sequences of composable morphisms. By \cite{IT13}, signed exceptional sequences are in bijection with ordered clusters tilting sets. Instead of $(n+1)^{n-1}$ exceptional sequences (in type $A_n$) we get $n!$ times the Catalan number $C_{n+1}=\frac1{n+2}\binom{2n+2}{n+1}$ signed exceptional sequences. (Theorem \ref{thm: IT13: signed exc seq} below.)

In this paper we construct a very simple 1-1 correspondence between exceptional sequences for the linear $A_n$ quiver $Q:1\to 2\to \cdots\to n$ and rooted labeled forests with $n$ vertices. This bijection has many nice applications including interpretation of relatively projective and injective modules. Each exceptional sequence is represented by a planar graph which makes many algebraic concepts clearly visible. The correspondence is given by the following theorem in which we represent modules by their supports. 

An indecomposable representation of a linear quiver of type $A_n$ is given by its support which is a closed interval $[a,b]$ where $1\le a\le b\le n$. The representation with this support will be labeled $M_{ab}$. This module has top the simple module at $a$ and socle the simple module at $b$.

\begin{customthm}{A1}\label{thm A1}
(Theorem \ref{theorem A1}) For each complete exceptional sequences $(E_1,\cdots,E_n)$ for $A_n$ consider the partial ordering on the set $\{v_1,v_2,\cdots,v_n\}$ given by $v_i<v_j$ if the support of $E_i$ is contained in the support of $E_j$. The Hasse diagram on this poset is a rooted labeled forest. Conversely, any rooted forest with $n$ vertices labeled $v_1,\cdots,v_n$ is the Hasse diagram for a unique exceptional sequence. 
\end{customthm}

One key result of this paper is that the location of the relatively projective and relatively injective objects in an exceptional sequence are immediately visible in the labeled forest as descending and ascending vertices in the forest (plus the roots which are both relatively projective and relatively injective). More precisely:

\begin{customthm}{A2}\label{thm A2}
(Theorem \ref{theorem A2}) In a complete exceptional sequence for linear $A_n$, $E_i$ is relatively projective if and only if either the corresponding node $v_i$ is a root of a tree in the forest (i.e., the support of $E_i$ is maximal) or $i<j$ for $E_j$ the smallest object containing $E_i$ in its support. Dually, $E_i$ is relatively injective if and only if either $v_i$ is a root in the forest or $i>j$ for $E_j$ the smallest object containing $E_i$ in its support.
\end{customthm}

Theorems \ref{thm A1} and \ref{thm A2} are illustrated in Figure \ref{Fig: introduction}. The labeled forest shows that there are 5 relatively projective objects in this exceptional sequence: $E_2,E_6,E_1,E_4,E_5$. The first two correspond to the roots $v_2,v_6$ of the forest and the last three correspond to the vertices $v_1,v_4,v_5$ whose indices $1,4,5$ are smaller than the indices $2,6,6$ of their parents $v_2,v_6,v_6$ resp.

\begin{figure}[htbp]
\begin{center}
\begin{tikzpicture}[scale=.75]
\draw[thick] (-7,1)--(-6,2.5)--(-5,1) (-11,1)--(-10,3)--(-13,1) (-9,1)--(-10,3);
\foreach \x in {(-7,1), (-6,2.5), (-5,1),(-11,1),(-13,1),(-10,3),(-9,1)}
\draw[fill,white] \x circle[radius=5mm];
\draw[thick] (-7,1) node{$v_1$} circle[radius=5mm];
\draw[thick] (-6,2.5) node{$v_2$} circle[radius=5mm];
\draw[thick] (-5,1) node{$v_3$} circle[radius=5mm];
\draw[thick] (-11,1) node{$v_4$} circle[radius=5mm];
\draw[thick] (-13,1) node{$v_5$} circle[radius=5mm];
\draw[thick] (-10,3) node{$v_6$} circle[radius=5mm];
\draw[thick] (-9,1) node{$v_7$} circle[radius=5mm];
\end{tikzpicture}
\caption{By Theorem \ref{thm A1}, this figure indicates the rooted labeled forest corresponding to the complete exceptional sequence for the quiver $A_7$:
\[
(E_1,E_2,E_3,E_4,E_5,E_6,E_7) = (M_{33} , M_{13} , M_{11} , M_{66} , M_{77} , M_{47} , M_{44})\qquad\qquad\qquad
\]
$M_{ab}$ denotes the module with support $[a,b]$. E.g., $M_{13}$ contains $M_{11}$ and $M_{33}$ in its support and $M_{aa}=S_a$ is the simple module at $a$.
}
\label{Fig: introduction}
\end{center}
\end{figure}

One consequence of Theorem \ref{thm A2} is the following.

\begin{customcor}{B}\label{cor B}(Corollary \ref{cor: C in text})
Every object in a complete exceptional sequence for linear $A_n$ is either relatively projective or relatively injective and at least one object is both.
\end{customcor}

Also, there is a three variable generating function for ascending and descending vertices in rooted labeled forests which is directly applicable to the corresponding exceptional sequence, giving a generating function for exceptional sequences:

\begin{customthm}{C}\label{cor C}(Theorem \ref{thm: main theorem})
We have a three variable generating function
\[
	P_n(a,b,c)=\sum a^pb^qc^r=c(a+(n-1)b+c)(2a+(n-2)b+c)\cdots ((n-1)a+b+c)
\]
where the sum is over all complete exceptional sequences $E_\ast$ for a linear quiver of type $A_n$ and the monomial $a^p b^q c^r$ is given by letting

$p=$ number of relatively projective but not relatively injective objects in $E_\ast$,

$q=$ number of relatively injective but not relatively projective objects in $E_\ast$,

$r=$ number of objects in $E_\ast$ which are both relatively projective and relatively injective.
\end{customthm}

We have the added bonus of visualizing the action of the braid group on exceptional sequences. Crawley-Boevey showed \cite{Crawley-Boevey} that the braid group acts transitively on the set of exceptional sequences for any acyclic quiver, and Ringel \cite{RingelExcSeq} extended this to all hereditary algebras. Here we show (in Figures \ref{fig: braid example}, \ref{fig: Braid example 2}, \ref{fig: A2 example}, \ref{fig: A3 example 1}) how to visualize this action using rooted labeled forests. This visualization is good enough to see the action of the Garside element of the braid group ($\Delta$). This is important for several reasons. 

For example, $\Delta^2$ generates the center of the braid group. For us, $\Delta$ is important since it gives a formula for the correspondence (proved in \cite{IT13}) between ``support tilting sets'' and ``signed exceptional sequences''. (See Proposition \ref{thm: clusters and c-vectors}). We give a detailed description of the action of $\Delta$ on rooted labeled forests in section \ref{ss: Garside}. Basically, $\Delta$ takes the projective vertices of a forest (on the left side), moves them to the top (to form the new roots) and moves the roots to the right side, to form the new injective vertices. (Propositions \ref{prop: properties of Delta F}, \ref{prop: properties of C F}, Figure \ref{fig: example of Delta on A7}). Thus $\Delta^2$ converts projective vertices of a forest into injective vertices (Figure \ref{fig: Delta 2}). Figure \ref{Fig: introduction} above has two projective vertices $v_5,v_6$ on the left, two roots $v_6,v_2$ at the top and two injective vertices $v_2,v_3$ on the right.

It has been pointed out to us that the braid group action on exceptional sequences already has a visualization given by parking functions in \cite{GG}. In section \ref{sec: parking functions} we review the bijection between exceptional sequences for linear $A_n$ and parking functions and the corresponding rooted labeled forests using ``Pr\"ufer codes''. We show using an example that this correspondence between exceptional sequences and rooted labeled forests is very different from ours.

There is another well-known visualization of exceptional sequences of type $A_n$ (with any orientation) given by ``chord diagrams'' which can be converted into labeled trees with $n+1$ vertices which are equivalent to rooted labeled forests with $n$ vertices. That construction, from \cite{GY}, is detailed in section \ref{ss: Goulden-Yong}. However, that construction does not keep track of relatively projective and injective objects which are the focus of this paper.

In conclusion, this paper gives an easy visualization of exceptional sequences using planar diagrams called ``rooted labeled forests'' (section \ref{sec 2}) and these forests are used to visualize the action of the braid group on exceptional sequences (section \ref{sec 3}), in particular the action of the important Garside element $\Delta$ (section \ref{sec 4}), and to show that exceptional sequences have a three variable generating function for its relatively projective and injective objects, the same as for rooted labeled forests.

\section{
Rooted labeled forests}\label{sec 2}\label{rooted labeled forest}

We will construct a 1-1 correspondence between complete exceptional sequences for linear $A_n$ and rooted labeled forest. Then, we derive some representation theoretic consequences.

\subsection{The Hasse diagram of an exceptional sequence}

We consider the Hasse diagram of an exceptional sequence and derive some of its properties. For example, we show that it is a rooted forest.

\begin{defn}\label{def: rooted labeled forest}
A \emph{rooted labeled forest} is a forest with vertices numbered $1$ through $n$ and a root chosen for each component. By adding an extra vertex labeled 0 and attaching it to each of the roots, such a structure is equivalent to a tree with $n+1$ vertices labeled $0$ through $n$.
\end{defn}

\begin{rem}\label{rem: when Hasse diagram is a rooted forest} The nodes of a rooted tree are partially ordered by paths to the root. Conversely, the Hasse diagram of a finite poset is a rooted forest if it has the property that each node has at most one parent. In that case, the maximal elements are roots of disjoint trees.
\end{rem}

Let $E_\ast=(E_1,\cdots,E_n)$ be a complete exceptional sequence for $A_n$. (Unless otherwise noted, we always take the orientation $1\to 2\to \cdots\to n$.) We define a partial ordering on the objects $E_i$ by $E_i\le E_j$ if the support of $E_i$ is contained in the support of $E_{j}$. In other words, $a_j\le a_i\le b_i\le b_j$ where $[a_i,b_i]$ denotes the support of $E_i$. We claim that the Hasse diagram of this partial ordering is a rooted labeled forest. This follow from the following lemma.

\begin{lem}\label{first lemma}
Given two objects $E_i=M_{a_i,b_i}$ and $E_j=M_{a_j,b_j}$ in an exceptional sequence, the intervals $[a_i,b_i]$, $[a_j,b_j]$ are ``noncrossing'', i.e. these closed interval are either disjoint or one is contained in the other.
\end{lem}

\begin{proof}
Given two objects whose supports cross, say, $M_{ab}$, $M_{cd}$ where $a<c\le b<d$, we have $\Hom(M_{cd},M_{ab})\neq0$ since $M_{cb}$ is a quotient module of $M_{cd}$ and a submodule of $M_{ab}$. Also, $\Ext(M_{ab},M_{cd})\neq0$ since we have a nonsplit exact sequence
\[
	0\to M_{cd}\to M_{cb}\oplus M_{ad}\to M_{ab}\to 0.
\]
So, both pairs $(M_{ab},M_{cd})$ and $(M_{cd},M_{ab})$ are not exceptional. Therefore, they cannot occur in an exceptional sequence in either order.
\end{proof}

\begin{lem}\label{second lemma} The set of elements above $E_i$ form a linearly ordered set.
\end{lem}

\begin{proof}
Suppose $E_i\le E_j,E_k$. Then the supports of $E_j,E_k$ meet and therefore, by Lemma \ref{first lemma}, one is contained in the other. So, they are linearly ordered.
\end{proof}

\begin{prop}\label{prop: Hasse diagram is a rooted forest}
The Hasse diagram of the objects $E_i$ in an exceptional sequence form a rooted labeled forest. Moreover, the roots are the maximal elements.
\end{prop}

\begin{proof}
By Lemma \ref{second lemma} a node cannot have two parents. Therefore, each connected component of the Hasse diagram is a tree with root at the top.
\end{proof}

\begin{eg}\label{eg: A7 Hasse diagram}
Consider the exceptional sequence for $A_7$:
\[
	(S_4, S_5, M_{15}, M_{67}, S_6, M_{12}, S_1)
\]
The  corresponding Hasse diagram is:
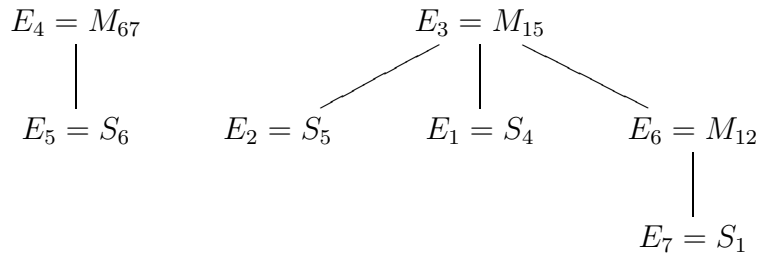
\[
\xymatrix{
E_4=M_{67}\ar@{-}[d] &&E_3=M_{15} \ar@{-}[dl] \ar@{-}[d]\ar@{-}[dr]\\
E_5=S_6 &E_2=S_5  & E_1=S_4   & E_6=M_{12}\ar@{-}[d]\\
	&& &
	E_7=S_1
	}
\]
We note that the leaves are simple modules. 
\end{eg}

\begin{rem}\label{rem: Hasse diagram in AR quiver}
The labeled forest can be drawn more systematically by choosing the vertices to lie in the Auslander-Reiten quiver of $\Lambda=KQ$. This is a planar diagram where the module $M_{ab}$ for $1\le a\le b\le n$ is placed at the point $(-b-a,b-a+1)$. For Example \ref{eg: A7 Hasse diagram} this is:
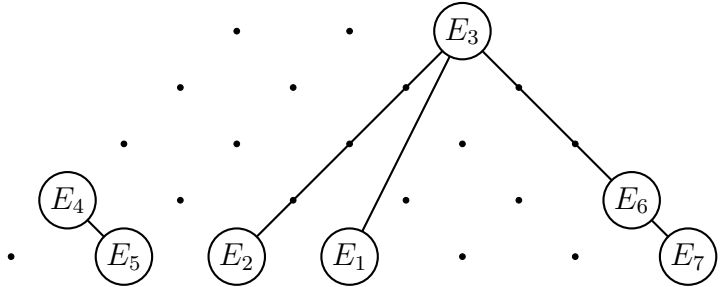
\[
\begin{tikzpicture}[scale=.75]
\draw[thick] (-12,1)--(-13,2) (-10,1)--(-6,5)--(-8,1)(-2,1)--(-6,5);
\foreach \x in {-14,-12,...,-2}\draw[fill] (\x,1) circle[radius=.5mm];
\foreach \x in {-13,-11,...,-3}
\draw[fill] (\x,2) circle[radius=.5mm];
\foreach \x in {-12,-10,...,-4}
\draw[fill] (\x,3) circle[radius=.5mm];
\foreach \x in {-11,-9,...,-5}
\draw[fill] (\x,4) circle[radius=.5mm];
\foreach \x in {-10,-8,-6}
\draw[fill] (\x,5) circle[radius=.5mm];
\foreach \x in {-9,-7}
\draw[fill] (\x,6) circle[radius=.5mm];
\draw[fill] (-8,7) circle[radius=.5mm];
\foreach \x in {(-12,1), (-2,1), (-8,1), (-10,1)}
\draw[fill,white] \x circle[radius=5mm];
\foreach \y in {(-13,2), (-6,5), (-3,2)}
\draw[fill,white] \y circle[radius=5mm];
\draw[thick] (-12,1) node{$E_5$} circle[radius=5mm];
\draw[thick] (-2,1) node{$E_7$} circle[radius=5mm];
\draw[thick] (-8,1) node{$E_1$} circle[radius=5mm];
\draw[thick] (-10,1) node{$E_2$} circle[radius=5mm];
\draw[thick] (-13,2) node{$E_{4}$} circle[radius=5mm];
\draw[thick] (-6,5) node{$E_{3}$} circle[radius=5mm];
\draw[thick] (-3,2) node{$E_{6}$} circle[radius=5mm];
\end{tikzpicture}
\]
Note that the $y$-coordinate of the point $M_{ab}$ is its length $b-a+1$.

It is easy to see that, with these coordinates, the Hasse diagram is embedded. Indeed, if $A$ is above $X$ and $B$ is above $Y$ and if the edges $AX$, $BY$ were to cross, then the supports of $A,B$ intersect. So, one of these supports contains the other, say $A$ is below $B$. Then $X$ and $Y$ would both be below $A$ and there would be no edge from $Y$ to $B$.
\end{rem}

\begin{prop}\label{prop: length=weight}
The length of the object $E_i=M_{a_i,b_i}$ (given by $b_i-a_i+1$) is the number of objects $E_j$ which are $\le E_i$ in the partial ordering. In particular, the leaves (the minimal elements) have length 1, i.e., they are simple modules.
\end{prop}

Given a rooted labeled tree, define the \emph{weight} of any node $v_i$ to be the number of nodes which are $\le v_i$. The statement of the proposition is that the weight of $v_i$ as a node in the Hasse diagram is equal to the length of the corresponding module $E_i$.

\begin{proof}[Proof of Proposition \ref{prop: length=weight}]
We use the fact that the dimension vectors $\undim E_i$ form a vector space basis for $\RR^n$. (See for example, \cite[Proposition 2.7]{IT13}.) Suppose that $E_i=M_{ab}$ with length $k=b-a+1$. Take the linear map $\pi:\RR^n\onto \RR^{k}$ given by projection onto coordinates $a$ through $b$. Then $\undim E_i$ maps to the vector $(1,1,\cdots,1)$. For any $E_j$ whose support is not contained in the interval $[a,b]$ we have by Lemma \ref{first lemma} that the support of $E_j$ is either disjoint from the interval $[a,b]$ or it contains that interval. In either case, $\pi(\undim E_j)$ is a scalar multiple of $\undim E_i$. Therefore, the image of $\pi$ is spanned by $\pi(\undim E_j)$ for those $E_j$ with support in $[a,b]$. So, there must be at least $k$ of these. However, there are at most $k$ linearly independent vectors with support in $[a,b]$. Therefore, the number of $E_j$ with support in $[a,b]$ is exactly $k$, the length of $E_i$. So, the length of $E_i$ is the weight of $v_i$ as claimed.
\end{proof}

Proposition \ref{prop: Hasse diagram is a rooted forest} shows that every complete exceptional sequence gives rise to a rooted labeled forest, i.e., we obtain a mapping $H$
\[
	\left\{
		\text{complete exceptional sequences for $A_n$}
	\right\}
		\xrightarrow{H}
	\left\{
		\text{rooted labeled forests on $n$}
	\right\}.
\]
We will show that this map is a bijection by constructing an inverse. Since the two sets have the same cardinality it will suffice to construct a right inverse to $H$. So, for every rooted labeled forest on $n$ we will construct the unique exceptional sequence having that forest as Hasse diagram. We will not prove uniqueness directly. Uniqueness follows from the counting argument.

\subsection{The exceptional sequence of a rooted labeled forests}

We give an explicit construction of the necessarily unique complete exceptional sequence for any rooted labeled forest. The construction will be recursive. 

For $1\le i\le n$, let $v_i$ be the node labeled $i$ in the forest $F$. Let $T_i$ be the tree having root $v_i$, i.e., $T_i$ consists of all nodes $v_j\le v_i$ in $F$. Then $v_j\le v_i$ if and only if $T_j\subset T_i$.

First, consider the connected case, i.e., $F$ is a tree. Let $v_r$ be the root of $F$, i.e., $F=T_r$. Let $v_{i_1},v_{i_2},\cdots,v_{i_k}$ be the children of the root $v_r$. Since the subscripts $i_j$ are distinct integers different from $r$ we may choose the labels so that
\[
	i_1<i_2<\cdots<i_p<r<i_{p+1}<\cdots<i_k
\]
where $0\le p\le k$. Recall that the weight $w(v_i)$ is the size of $T_i$ by definition. In particular, the weight of $v_r$ is $n$ and \[
	n=1+\sum w(v_{i_j})
\]

We assign the module $E_r:=M_{1n}$ to $v_r$. Since $M_{1n}$ is uniserial, it has a filtration with $k+1$ subquotients, one of length 1 and the others of length $w(v_{i_j})$. Denote the filtration:
\[
0=A_{p+1}\subset A_p\subset A_{p-1}\subset\cdots\subset A_2\subset A_1\subset B_k\subset \cdots\subset B_{p+1}\subset B_p=M_{1n}
\]
We may choose these so that $A_x/A_{x+1}$ has length $w(v_{i_x})$ when $x\le p$ and $B_{y-1}/B_y$ has length $w(v_{i_y})$ for $y\ge p+1$. (This makes $B_k/A_1$ simple, of length 1.) Then we let
\[
E_{i_x}=A_x/A_{x+1},\qquad E_{i_y}=B_{y-1}/B_y.
\]

\begin{prop}\label{prop: children of the root}
$(E_{i_1},\cdots,E_{i_p}, E_r, E_{i_{p+1}},\cdots,E_{i_k})$ is an exceptional sequence.
\end{prop}

\begin{proof}
It suffices to show that $(X,Y)$ is an exceptional pair for any subsequence $X,Y$ of length 2.

Case 1: $X=E_r$. Then $Y=E_{i_y}$ for $y>p$ which has no morphism to $E_r$ since its support does not contain $n$ by construction. $\Ext(Y,E_r)=0$ since $E_r$ is injective.

Case 2: $Y =E_r$, Then $X=E_{i_x}$ for $x\le p$. These cannot be quotients of $Y$. Hence there is no morphism $Y\to X$ and $\Ext(E_r,X)=0$ since $E_r$ is projective.

For the remaining pairs, there is no morphism since they have disjoint supports.

Case 3: $X=E_{i_x}$ where $x\le p$ and $Y=E_{i_y}$ where $y\ge p+1$. Then the supports of $X,Y$ are disjoint and not consecutive. Therefore there is no morphism and no extension.

In the remaining cases, $X,Y$ are on the same side of $E_r$. So, their supports are disjoint and $\Ext(Y,X)=0$ since, in any extension, $Y$ would be the submodule.
\end{proof}

We come to the recursion step. Take any subtree of $F$ with root $v_s$. Suppose the corresponding module has been constructed, say $E_s=M_{ab}$ with length $b-a+1=w(v_s)$. In particular $E_s$ will be simple when $v_s$ is a leaf. When $v_s$ is not a leaf, we repeat the construction above for the children $v_{j_z}$ of $v_s$. This will assign modules $E_{j_z}$ of length equal to the weight $w(v_{j_z})$ having supports which are closed disjoint subintervals of $[a,b]$ with union $[a,b]$ minus one point which we call the ``gap''. (In the construction above, the gap is the support of the simple module $B_k/A_1$.) Continuing in this way we construct all modules $E_t$ in the case when $F$ is connected.

Now consider the case when $F$ is a disjoint union of trees $T_1,\cdots,T_k$ with roots $v_{j_1},\cdots,v_{j_k}$ where $j_1<j_2<\cdots<j_k$. Then the weights $w(v_{j_z})$ add up to $n$. Let $E_{j_1},\cdots,E_{j_k}$ be the modules with disjoint supports of length $w(v_{j_z})$ which are consecutive. For instance $E_{j_1}=M_{1x}$ where $x=w(v_{j_1})$, $E_{j_2}=M_{x+1,y}$ where $y=x+w(v_{j_2})$, etc.

\begin{prop}\label{prop: F disconnected case}
$(E_{j_1},\cdots,E_{j_k})$ is an exceptional sequence.
\end{prop}

\begin{proof}
There are no morphisms between these objects since their supports are disjoint. For any $x<y$, $\Ext (E_{j_y},E_{j_x})=0$ since in any extension, $E_{j_y}$ will be the submodule. 
\end{proof}

\begin{rem}\label{rem 2.10: consecutive supports}
We say that a sequence of modules $X_1,\cdots,X_k$ have \emph{consecutive supports} if $X_i=M_{a_i,b_i}$ where each $a_i=b_{i-1}+1$. For example, the modules in Proposition \ref{prop: F disconnected case} have consecutive supports. Modules with consecutive support  form an exceptional sequence.
\end{rem}

We list the basic properties of the modules $E_i$ corresponding to the nodes $v_i$ under the construction above.

\begin{prop}\label{prop: properties of Ei}
Given any rooted labeled forest $F$, let $E_i$ be the module corresponding to node $v_i$ by the above recursive construction. Then
\begin{enumerate}
\item Each $E_i$ has length equal to $w(v_i)$.
\item The support of $E_i$ is contained in the support of $E_j$ if and only if $v_i\le v_j$.
\item The supports of $E_i,E_j$ are disjoint if and only if $v_i,v_j$ are not comparable.
\item If $v_{i_1},\cdots,v_{i_k}$ are the children of $v_r$ and $i_1<i_2<\cdots<i_p<r<i_{p+1}<\cdots<i_k$ then 
	\begin{enumerate}
	\item $E_{i_x}$ is a submodule of $E_r$ if and only if $x=p$, i.e., if $v_{i_x}$ is the child of $v_r$ with the largest label less than $r$. (None of the $E_{i_x}$ are submodules of $E_r$ when $p=0$.)
	\item $E_{i_y}$ is a quotient of $E_r$ if and only if $y=p+1$, i.e., if $v_{i_y}$ is the child of $v_r$ with the smallest label greater than $r$. (None of the $E_{i_y}$ are quotients of $E_r$ when $p=k$.)
	\item $E_{i_1},\cdots,E_{i_p}$ have consecutive supports.
	\item $E_{i_{p+1}},\cdots,E_{i_k}$ have consecutive supports.
	\end{enumerate}
\end{enumerate}
\end{prop}

\begin{proof}
Property (1) is Proposition \ref{prop: length=weight}. Property (4) follow from Proposition \ref{prop: children of the root}. To prove (2) and (3) we note that, for any $i,j$ there are only two possibilities: either one is $\le$ the other or they are not comparable. In the first case, suppose $v_i<v_j$. Then there is a sequence $v_i=v_{k_0}<v_{k_1}<v_{k_2}<\cdots<v_j$ where each $v_{k_x}$ is a child of $v_{k_{x+1}}$ in that case, $\supp E_{k_x}\subset \supp E_{k_{x+1}}$ by construction. So, $\supp E_i\subset \supp E_j$. In the second case, either $v_i,v_j$ lie in disjoint trees with distinct roots $v_k,v_\ell$ in which case $E_k$ and $E_\ell$ have disjoint supports by Proposition \ref{prop: F disconnected case}. Otherwise, $v_i,v_j$ lie in the same tree. Let $v_k$ be the smallest node $\ge v_i,v_j$. Then $v_i,v_j$ are $\le$ two children of $v_k$ where the corresponding modules have disjoint supports. So, $E_i,E_j$ have disjoint supports.
\end{proof}

\begin{thm}{(Theorem \ref{thm A1})}\label{thm: exceptional sequence of a forest}\label{theorem A1}
 Given any rooted labeled forest $F$, let $E_i$ be the module corresponding to node $v_i$ by the above recursive construction. Then
\[
	(E_1,E_2,\cdots,E_n)
\]
is a complete exceptional sequence.
\end{thm}

To prove this theorem we consider the three cases when $E_i$ is a submodule of $E_j$, $E_i$ is a quotient module of $E_j$ and $E_i,E_j$ extend each other.

\begin{lem}\label{lem1: Ei submodule of Ej}
If $E_i$ is a submodule of $E_j$ and $i\neq j$ then $v_i<v_j$ and the nodes $v_{k_1},\cdots,v_{k_t}$ connecting $v_i$ to $v_j$ in $F$ have increasing labels: $i<k_1<k_2<\cdots<k_t<j$.
\end{lem}

\begin{proof} Since $v_i<v_{k_t}<v_j$, $\supp E_i\subset \supp E_{k_t}\subset \supp E_j$. Since $E_i$ contains the socle of $E_j$, so does $E_{k_t}$. So, $E_{k_t}$ is a submodule of $E_j$ and $v_{k_t}$ is a child of $v_j$. It follows that the label must be smaller by the construction: $k_t<j$. By the same argument $k_p<k_{p+1}$ for all $p$ and $i<k_1$.
\end{proof}

\begin{lem}\label{lem2: Ei quotient of Ej}
If $E_i$ is a quotient module of $E_j$ and $i\neq j$ then $v_i<v_j$ and the nodes $v_{k_1},\cdots,v_{k_t}$ connecting $v_i$ to $v_j$ in $F$ have decreasing labels: $i>k_1>k_2>\cdots>k_t>j$.
\end{lem}

\begin{proof}
This statement is dual to Lemma \ref{lem1: Ei submodule of Ej} and has an analogous proof.
\end{proof}

\begin{lem}\label{lem3: Ei, Ej consecutive}
$\Ext(E_i,E_j)\neq0$ if and only if $E_i,E_j$ have (disjoint) consecutive supports and, in that case, $E_i$ is a submodule of some $E_k$ (which may be $E_i$) and $E_j$ is a quotient module of some $E_\ell$ (which may be $E_j$) where $v_k,v_\ell$ are sibling or they are both roots and
\[
	i\le k<\ell \le j.
\]
\end{lem}

\begin{proof}
Let $\Ext(E_i,E_j)\neq0$. Then either the supports are consecutive or they overlap. But the second case is excluded by Proposition \ref{prop: properties of Ei} (2), (3). So, $E_i,E_j$ have consecutive supports. The converse is clear.

Suppose the supports of $E_i,E_j$ are consecutive, say $\supp E_i=[a,b]$, $\supp E_j=[b+1,c]$. Let $v_k$ be maximal in $F$ so that $v_i\le v_k$ and $v_j\not\le v_k$. Then $E_k$ contains the support of $E_i$ and is disjoint from the support of $E_j$. So, $b\in\supp E_k$ and $b+1\notin \supp E_k$ making $E_i$ a submodule of $E_k$. By Lemma \ref{lem1: Ei submodule of Ej}, $i\le k$.

Similarly, let $v_\ell$ be maximal in $F$ so that $v_j\le v_\ell$ and $v_i\not\le v_\ell$. Then $\ell\le j$ by Lemma \ref{lem2: Ei quotient of Ej}. By maximality of $v_k,v_\ell$, they must be sibling (or roots of disjoint trees). In either case, $k<\ell$ since $E_k,E_\ell$ have consecutive supports. (This follows from either Proposition \ref{prop: properties of Ei}(4) or Proposition \ref{prop: F disconnected case}.) This proves the lemma.
\end{proof}

\begin{proof}[Proof of Theorem \ref{thm: exceptional sequence of a forest}]
These lemmas imply that $(E_1,\cdots,E_n)$ is a complete exceptional sequence. Indeed, if $i<j$ we have $\Hom(E_j,E_i)=0$ since $E_j$ is not a submodule of $E_i$ by Lemma \ref{lem1: Ei submodule of Ej} and $E_i$ is not a quotient module of $E_j$ by Lemma \ref{lem2: Ei quotient of Ej}. Also, $\Ext(E_j,E_i)=0$ by Lemma \ref{lem3: Ei, Ej consecutive}.
\end{proof}

\begin{eg}\label{eg: rooted forest for A9}
It is easy to convert a rooted labeled forest into a complete exceptional sequence. Start with the following.
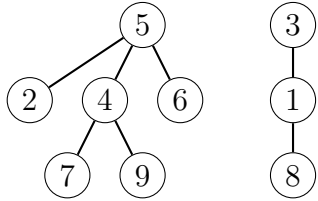
\begin{center}
\begin{tikzpicture}
\coordinate (A3) at (5.5,2);
\coordinate (A1) at (5.5,1);
\coordinate (A8) at (5.5,0);
\coordinate (B2) at (2,1);
\coordinate (B4) at (3,1);
\coordinate (B6) at (4,1);
\coordinate (B7) at (2.5,0);
\coordinate (B9) at (3.5,0);
\coordinate (B5) at (3.5,2);

\draw[thick] (A8)--(A1)--(A3);
\draw[thick] (B2)--(B5)--(B4)--(B7) (B9)--(B4) (B6)--(B5);
\draw[white,fill] (A1) circle[radius=3mm];
\draw (A1) circle[radius=3mm];
\draw (A1) node{1};
\draw[white,fill] (A3) circle[radius=3mm];
\draw (A3) circle[radius=3mm];
\draw (A3) node{3};
\draw[white,fill] (B2) circle[radius=3mm];
\draw (B2) circle[radius=3mm];
\draw (B2) node{2};
\draw[white,fill] (B4) circle[radius=3mm];
\draw (B4) circle[radius=3mm];
\draw (B4) node{4};
\draw[white,fill] (B5) circle[radius=3mm];
\draw (B5) circle[radius=3mm];
\draw (B5) node{5};
\draw[white,fill] (B6) circle[radius=3mm];
\draw (B6) circle[radius=3mm];
\draw (B6) node{6};
\draw[white,fill] (B7) circle[radius=3mm];
\draw (B7) circle[radius=3mm];
\draw (B7) node{7};
\draw[white,fill] (A8) circle[radius=3mm];
\draw (A8) circle[radius=3mm];
\draw (A8) node{8};
\draw[white,fill] (B9) circle[radius=3mm];
\draw (B9) circle[radius=3mm];
\draw (B9) node{9};
\end{tikzpicture}
\end{center}

By Remark \ref{rem 2.10: consecutive supports}, the modules $E_3,E_5$ corresponding to the roots $v_3,v_5$ have consecutive supports. By Proposition \ref{prop: properties of Ei} (1), these supports have lengths 3,6. So, $E_3=M_{13}$ and $E_5=M_{49}$. Since $1<3$, $E_1$ is the length 2 submodule of $E_3$. So, $E_1=M_{23}$. By Proposition \ref{prop: properties of Ei}(4), $E_6$ has support at $4$ followed by a gap, then $E_2,E_4$ have consecutive supports in the interval $[6,9]$. So, $E_6=S_4$, $E_2=S_6$ and $E_4=M_{79}$. The children of $v_4$ are leaves with larger labels, so the corresponding modules are simples of injective type: $E_7=S_7$, $E_9=S_8$. Similarly, $E_8=S_2$. The complete exceptional sequence is
\[
	E_\ast=(M_{23},S_6,M_{13}, M_{79}, M_{49}, S_4, S_7, S_2, S_8)
\]
This same example is done using chord diagrams following \cite{GY} in Figure \ref{fig: Goulden-Yong}.
\end{eg}

\subsection{Relatively projective and injective objects}\label{ss: rel proj and rel inj}

\begin{defn}\label{def: relatively projective and perp cat}
An object $E$ of a subcategory $\cC$ of $mod$-$\Lambda$ is called \emph{relatively projective}, respectively \emph{relatively injective} if it is a projective, respectively injective, object of the subcategory $\cC$. The \emph{right perpendicular category} $Z^\perp$ of any $\Lambda$-module $Z$ is defined to be the full subcategory of $mod$-$\Lambda$ of all modules $Y$ so that
\[
	\Hom_\Lambda(Z,Y)=0=\Ext^1_\Lambda(Z,Y).
\]
Similarly, the \emph{left perpendicular category} $^\perp Z$ of any $\Lambda$-module $Z$ is defined to be the full subcategory of $mod$-$\Lambda$ of all modules $X$ so that
\[
	\Hom_\Lambda(X,Z)=0=\Ext^1_\Lambda(X,Z).
\]

By abuse of terminology, we call an object $E_i$ in a complete exceptional sequence $(E_1,\cdots,E_n)$ \emph{relatively projective} if it is a projective object of the \emph{right perpendicular} category $Z^\perp$ of $Z=E_{i+1}\oplus \cdots\oplus E_n$. Similarly, we call an object $E_i$ in a complete exceptional sequence $(E_1,\cdots,E_n)$ \emph{relatively injective} if it is an injective object of the \emph{left perpendicular} category $^\perp Z$ of $Z=E_1\oplus \cdots\oplus E_{i-1}$. 
\end{defn}

\begin{prop}\label{prop: perpendicular category is abelian}
If any two of the objects in a short exact sequence $0\to A\to B\to C\to0$ lie in $Z^\perp$ (or $\,^\perp Z$) then so does the third.
\end{prop}

\begin{proof}
This follows from the $6$-term exact sequence:
\[
\begin{split}
	0\to &\Hom(Z,A)\to \Hom(Z,B)\to \Hom(Z,C) \\
	&\to \Ext(Z,A)\to \Ext(Z,B)\to \Ext(Z,C)\to 0.
	\end{split}
\]
An analogous argument works for $\,^\perp Z$.
\end{proof}

Since $X\in Y^\perp$ if and only if $Y\in \,^\perp X$, we get the following.

\begin{cor}\label{cor: A+B perp is contained in C perp}
For any short exact sequence $0\to A\to B\to C\to 0$, $(A\oplus B)^\perp=A^\perp\cap B^\perp$ is contained in $C^\perp$ and similarly, $\,^\perp(A\oplus B)\subseteq \,^\perp C$.\qed
\end{cor}

\begin{thm}{(Theorem \ref{thm A2})}\label{theorem A2}
Let $E_\ast=(E_1,\cdots,E_n)$ be a complete exceptional sequence for a linear quiver of type $A_n$. Let $F$ be the corresponding rooted labeled forest. For each $E_i$ in $E_\ast$, let $v_i$ be the corresponding node of $F$. Then, 
\begin{enumerate}
\item $E_i$ is relatively projective and relatively injective if and only if $v_i$ is a root of $F$.
\item For $v_i$ not a root of $F$, let $v_j$ be the parent of $v_i$. Then, $E_i$ is relatively projective, resp. injective in $E_\ast$ if and only if $i<j$, resp $i>j$.
\end{enumerate}
\end{thm}

\begin{cor}\label{cor: C in text}(Corollary \ref{cor B})
Every object in a complete exceptional sequence of type $A_n$ with straight orientation is either relatively projective or relatively injective and at least one object is both.\qed
\end{cor}

To prove Theorem \ref{theorem A2}, let $v_{i_1},\cdots,v_{i_k}$ be the children of $v_j$ with
\[
	i_1<i_2<\cdots<i_p<j<i_{p+1}<\cdots<i_k.
\]
As in Proposition \ref{prop: children of the root}, we have that $E_j=M_{ab}$ and $E_{i_1},\cdots,E_{i_p}$ is a sequence of modules with consecutive support in the interval $[a+1,b]$ which is the support of $rM_{ab}$, the radical of $M_{ab}$. Also, $E_{i_p}$ is a submodule of $M_{ab}$. In more details, we have $E_{i_s}=M_{a_s,b_s}$ where $a\le b_0<b_1<b_2<\cdots<b_p=b$ and $a_s=b_{s-1}+1$ for $1\le s\le p$. In Example \ref{eg: A7 Hasse diagram}, $p=2$ and $a,b_0,b_1,b_2=1,3,4,5$ making $E_{i_1}=E_1=M_{44}$ and $E_{i_2}=E_2=M_{55}$.

\begin{lem}\label{lem: i>j then Ei is not relatively projective}
The objects $E_{i_y}$ for $y>p$ are not relatively projective in $E_\ast$.
\end{lem}

\begin{proof}
Let $W$ be the direct sum of the objects $E_t$ which come after $E_{i_y}$, i.e., $t>i_y$. We will show that $E_t$ is not a projective object of $W^\perp$.

\underline{Claim}: For any $p<x\le y$, $W^\perp$ contains a module $M_{cb}$ with socle $S_b$ which maps onto $E_{i_x}$.

Since $E_{i_x}$ is a proper quotient of $M_{cb}$, $E_{i_x}$ cannot be projective in $W^\perp$ for $p<x\le y$. When $x=y$ this proves the lemma. So, it suffices to prove this claim.

The proof of the claim is by induction on $x$. For $x=p+1$, $E_{i_{p+1}}$ is a quotient of $E_j=M_{ab}$ which lies in $W^\perp$ since it is to the left of $E_{i_y}$ in $E_\ast$. We take $c=a$ in that case. Suppose the claim holds for $x<y$. Then, by induction, there is a module $M_{cb}\in W^\perp$ which maps onto $E_{i_x}$ which also lies in $W^\perp$. Therefore the kernel $M_{db}$ lies in the abelian category $W^\perp$. But $M_{db}$ maps onto $E_{i_{x+1}}$ since $E_{i_x},E_{i_{x+1}}$ have consecutive supports. So, the claim holds for $x+1$. This proves the claim and the lemma.
\end{proof}

\begin{lem}\label{lem: basic case}
$rM_{ab}$ is a projective object in the perpendicular category $M_{ab}^\perp$.
\end{lem}

\begin{proof}
When $b=n$, $rM_{ab}$ is projective and lies in $M_{ab}^\perp$.

For $b<n$, $rM_{ab}$ is a quotient of $\tau M_{ab}=M_{a+1,b+1}$. So, $\Hom(rM_{ab},\tau M_{ab})=0$ and $\Hom(M_{ab},rM_{ab})=0$. So, $rM_{ab}\in M_{ab}^\perp$. 

To show that $rM_{ab}$ is projective in $M_{ab}^\perp$, suppose not. Then there is another object  $X\in M_{ab}^\perp$ and an epimorphism $X\onto rM_{ab}$. But then $X=M_{a+1,c}$ for some $c>b$. Then we have an epimorphism $X\onto \tau M_{ab}$ contradicting the assumption that $X\in M_{ab}^\perp$. So, $rM_{ab}$ is projective in $M_{ab}^\perp$.
\end{proof}

Lemma \ref{lem: basic case} is the case $s=p$ of the following lemma:

\begin{lem}\label{lem: long projectives in W-perp}
$M_{a+1,b_s}$ is a projective object in the perpendicular category $W^\perp$ where $W=M_{ab}\oplus E_{i_{s+1}}\oplus\cdots\oplus E_{i_p}$.
\end{lem}

\begin{proof}
For $b_s=b$, $M_{a+1,b_s}=rM_{ab}$ and the previous lemma applies.

For $b_s<b$, $M_{a+1,b_s}$ is a quotient of $\tau M_{ab}=M_{a+1,b+1}$. So, $M_{a+1,b_s}\in M_{ab}^\perp$. Also, $M_{a+1,b_s}\in E_{i_t}^\perp$ for $s<t\le p$ since the support of $E_{i_t}$ lies in $(b_s,n]$. So, $M_{a+1,b_s}\in W^\perp$

The module $M_{a,b_s}$ is an iterated quotient of components of $W$ since $M_{a,b_{p-1}}=M_{ab}/E_{i_p}$ and $M_{a,b_s}=M_{a,b_{s+1}}/E_{i_{s+1}}$. Therefore, $W^\perp\subset M_{a,b_s}^\perp$ by Corollary \ref{cor: A+B perp is contained in C perp}. $M_{a+1,b_s}$ is projective in $M_{a,b_s}^\perp$ by the previous lemma and therefore it is projective in $W^\perp$.
\end{proof}

\begin{thm}\label{thm: B2a}
Let $v_i$ be a child of $v_j$ in the rooted labeled forest $F$. Then the object $E_j$ in the corresponding complete exceptional sequence $E_\ast=(E_1,\cdots,E_n)$ is relatively projective if and only if $i<j$.
\end{thm}

\begin{proof} If $i>j$ then, by Lemma \ref{lem: i>j then Ei is not relatively projective}, $E_i$ is not relatively projective. So, suppose $i<j$. Then, in the notation above, $i=i_s$ for some $s\le p$.

Let $W$ be the direct sum of all objects $E_k$ in the exceptional sequence with $k>i$. This includes the summands of $W_0=E_j \oplus E_{i_{s+1}}\oplus\cdots\oplus E_{i_p}$. Thus $W^\perp\subset W_0^\perp$. By Lemma \ref{lem: long projectives in W-perp}, $W_0^\perp$ contains as a projective object the module $M_{a+1,b_s}$ which has $E_{i_s}=M_{a_s,b_s}$ as a submodule. Since $W_0^\perp$ is a hereditary abelian category, this makes $E_{i_s}$ a projective object of $W_0^\perp$. So, $E_{i_s}$ is a projective object of $W^\perp\subset W_0^\perp$. 
\end{proof}

We have the following dual statement with an analogous argument.

\begin{thm}\label{thm: B2b}
Let $v_i$ be a child of $v_j$ in the rooted labeled forest $F$. Then the object $E_j$ in the corresponding complete exceptional sequence $E_\ast=(E_1,\cdots,E_n)$ is relatively injective if and only if $i>j$.\qed
\end{thm}

By the discussion above, the only $E_j$ which have a chance to be both relatively projective and relatively injective are ones corresponding to the roots of the components of $F$. Let $v_{j_1},v_{j_2},\cdots,v_{j_k}$ be the roots of $F$ with $j_1<j_2<\cdots<j_k$. Then the modules $E_{j_1},\cdots,E_{j_k}$ have consecutive supports with union $[1,n]$.

\begin{thm}\label{thm: B2c}
The objects
$E_{j_1},\cdots,E_{j_k}$ corresponding to the roots of the forest $F$ are relatively projective and relatively injective in the complete exceptional sequence $E_\ast$.
\end{thm}

\begin{proof}
First, $E_{j_k}$ is projective since it is a submodule of the projective module $M_{1n}$. Then, as in the proof of Lemma \ref{lem: long projectives in W-perp}, the quotient $M_{1n}/E_{j_k}$ is a projective object in the category $E_{j_k}^\perp$. Since $E_{j_{k-1}}$ is a submodule of $M_{1n}/E_{j_k}$, it also becomes projective in $E_{j_k}^\perp$ and therefore projective in $W^\perp$ where $W$ is the direct sum of all $E_t$ for $t>j_{k-1}$. Repeating this process, we see that each $E_{j_t}$ is relatively projective in $E_\ast$.

The dual arguments show that each $E_{j_t}$ is relatively injective in $E_\ast$.
\end{proof}

Theorem \ref{theorem A2} is given by combining Theorems \ref{thm: B2a}, \ref{thm: B2b} and \ref{thm: B2c}.

\subsection{Chord diagrams and rooted labeled trees}\label{ss: Goulden-Yong}

First, we note that a rooted forest with $n$ labeled vertices is equivalent to a rooted tree with $n+1$ vertices and $n$ labeled edges. (Just move the edge labels to the endpoint furthest from the root, delete the root of the tree, then declare the adjacent vertices to be roots of the resulting forest.)

A bijection between exceptional sequences on quivers of type $A_n$ (with any orientation) and rooted trees with $n$ labeled edges is given as follows. 
\begin{figure}[htbp]
\begin{center}
\begin{tikzpicture}
\coordinate (A1) at (0,3);
\coordinate (A2) at (1.76,2.43);
\coordinate (A3) at (2.85,.93);
\coordinate (A4) at (2.85,-.93);
\coordinate (A5) at (1.76,-2.43);
\coordinate (A6) at (0,-3);
\coordinate (A7) at (-1.76,-2.43);
\coordinate (A8) at (-2.85,-.93);
\coordinate (A9) at (-2.85,.93);
\coordinate (A0) at (-1.76,2.43);
\coordinate (B1) at (0,3.5);
\coordinate (B2) at (2.06,2.83);
\coordinate (B3) at (3.33,1.08);
\coordinate (B4) at (3.33,-1.08);
\coordinate (B5) at (2.06,-2.83);
\coordinate (B6) at (0,-3.5);
\coordinate (B0) at (-2.06,2.83);
\coordinate (B9) at (-3.33,1.08);
\coordinate (B8) at (-3.33,-1.08);
\coordinate (B7) at (-2.06,-2.83);
\draw (0,0) circle[radius=3cm];
\foreach \x in {A1,A2,A3,A4,A5,A6,A7,A8,A9,A0}
\draw[fill] (\x) circle[radius=2pt];
\draw (B1) node{1};
\draw (B2) node{2};
\draw (B3) node{3};
\draw (B4) node{4};
\draw (B5) node{5};
\draw (B6) node{6};
\draw (B7) node{7};
\draw (B8) node{8};
\draw (B9) node{9};
\draw (B0) node{0};
\coordinate (T) at (-.2,1);
\coordinate (T5) at (0,-1);
\coordinate (T4) at (-1.7,0);
\coordinate (T9) at (-3,0);
\coordinate (T7) at (-2.5,-2);
\coordinate (T2) at (-1,-3);
\coordinate (T6) at (3,-2);
\coordinate (T3) at (1,1);
\coordinate (T1) at (2,.7);
\coordinate (T8) at (3,2);

\draw[thick,red] (A2)..controls (.5,2) and (2,0)..(A4); 
\draw[thick,red] (A4)..controls (1,-1) and (0,0)..(A0); 
\coordinate (C) at (.6,-.3);
\coordinate (C2) at (1.3,1.5);
\draw[fill,white] (C2) circle[radius=2.5mm];
\draw[thick] (C2) circle[radius=2.5mm];
\draw (C2) node{1};
\begin{scope}
\clip (0,0) circle[radius=3cm];
\draw[blue] (T) node[above] {$\ast$};
\draw[very thick,blue] (T)--(T5)--(T4)--(T9) (T4)--(T7);
\draw[very thick,blue] (T2)--(T5)--(T6) (T)--(T3)--(T1)--(T8);
\draw[fill,white] (C) circle[radius=2.5mm];
\draw[thick] (C) circle[radius=2.5mm];
\draw (C) node{5};
\end{scope}
\begin{scope}[rotate=-144]
\clip (0,0) circle[radius=3cm];
\draw[thick,red] (-5,0) circle [radius=4.04cm];
\draw[fill,white] (-1,0) circle[radius=2.5mm];
\draw[thick] (-1,0) circle[radius=2.5mm];
\draw (-1,0) node{3};
\end{scope}
\begin{scope}
\clip (0,0) circle[radius=3cm];
\draw[thick,red] (-3.93,0) circle [radius=1.44cm];
\draw[fill,white] (-2.5,0) circle[radius=2.5mm];
\draw[thick] (-2.5,0) circle[radius=2.5mm];
\draw (-2.5,0) node{9};
\end{scope}
\begin{scope}[rotate=36]
\clip (0,0) circle[radius=3cm];
\draw[thick,red] (-3.93,0) circle [radius=1.44cm];
\draw[fill,white] (-2.5,0) circle[radius=2.5mm];
\draw[thick] (-2.5,0) circle[radius=2.5mm];
\draw (-2.5,0) node{7};
\end{scope}
\begin{scope}[rotate=72]
\clip (0,0) circle[radius=3cm];
\draw[thick,red] (-3.93,0) circle [radius=1.44cm];
\draw[fill,white] (-2.5,0) circle[radius=2.5mm];
\draw[thick] (-2.5,0) circle[radius=2.5mm];
\draw (-2.5,0) node{2};
\end{scope}
\begin{scope}[rotate=144]
\clip (0,0) circle[radius=3cm];
\draw[thick,red] (-3.93,0) circle [radius=1.44cm];
\draw[fill,white] (-2.5,0) circle[radius=2.5mm];
\draw[thick] (-2.5,0) circle[radius=2.5mm];
\draw (-2.5,0) node{6};
\end{scope}
\begin{scope}[rotate=-144]
\clip (0,0) circle[radius=3cm];
\draw[thick,red] (-3.93,0) circle [radius=1.44cm];
\draw[fill,white] (-2.5,0) circle[radius=2.5mm];
\draw[thick] (-2.5,0) circle[radius=2.5mm];
\draw (-2.5,0) node{8};
\end{scope}
\begin{scope}
\clip (0,0) circle[radius=3cm];
\draw[thick,red] (-5,0) circle [radius=4.04cm];
\draw[fill,white] (-1,0) circle[radius=2.5mm];
\draw[thick] (-1,0) circle[radius=2.5mm];
\draw (-1,0) node{4};
\end{scope}
\begin{scope}[xshift=8cm] 
\coordinate (T) at (-.2,1);
\coordinate (T5) at (0,-1);
\coordinate (T4) at (-1.7,0);
\coordinate (T9) at (-3,0);
\coordinate (T7) at (-2.5,-2);
\coordinate (T2) at (-1,-3);
\coordinate (T6) at (3,-2);
\coordinate (T3) at (1,1);
\coordinate (T1) at (2,.7);
\coordinate (T8) at (3,2);
\begin{scope}
\draw (T) node[above] {$\ast$};
\draw[thick,dashed] (T3)--(T)--(T5);
\draw[very thick,blue] (T5)--(T4)--(T9) (T4)--(T7);
\draw[very thick,blue] (T2)--(T5)--(T6) (T3)--(T1)--(T8);
\end{scope}
\draw[fill,white] (T1) circle[radius=2.5mm];
\draw[thick] (T1) circle[radius=2.5mm];
\draw (T1) node{1};
\draw[fill,white] (T2) circle[radius=2.5mm];
\draw[thick] (T2) circle[radius=2.5mm];
\draw (T2) node{2};
\draw[fill,white] (T3) circle[radius=2.5mm];
\draw[thick] (T3) circle[radius=2.5mm];
\draw (T3) node{3};
\draw[fill,white] (T4) circle[radius=2.5mm];
\draw[thick] (T4) circle[radius=2.5mm];
\draw (T4) node{4};
\draw[fill,white] (T5) circle[radius=2.5mm];
\draw[thick] (T5) circle[radius=2.5mm];
\draw (T5) node{5};
\draw[fill,white] (T6) circle[radius=2.5mm];
\draw[thick] (T6) circle[radius=2.5mm];
\draw (T6) node{6};
\draw[fill,white] (T7) circle[radius=2.5mm];
\draw[thick] (T7) circle[radius=2.5mm];
\draw (T7) node{7};
\draw[fill,white] (T8) circle[radius=2.5mm];
\draw[thick] (T8) circle[radius=2.5mm];
\draw (T8) node{8};
\draw[fill,white] (T9) circle[radius=2.5mm];
\draw[thick] (T9) circle[radius=2.5mm];
\draw (T9) node{9};
\end{scope}
\end{tikzpicture}
\caption{The exceptional sequence $(M_{23},S_6,M_{13}, M_{79}, M_{49}, S_4, S_7, S_2, S_8)$ from Example \ref{eg: rooted forest for A9} is drawn as the sequence of chords $(24,67,14,70,40,45,78,23,89)$ in red. The corresponding tree with root $(\ast)$ in the region containing the arc $0$ to $1$ is indicated in blue. Deleting the root gives the rooted labeled forest in Example \ref{eg: rooted forest for A9}.}
\label{fig: Goulden-Yong}
\end{center}
\end{figure}
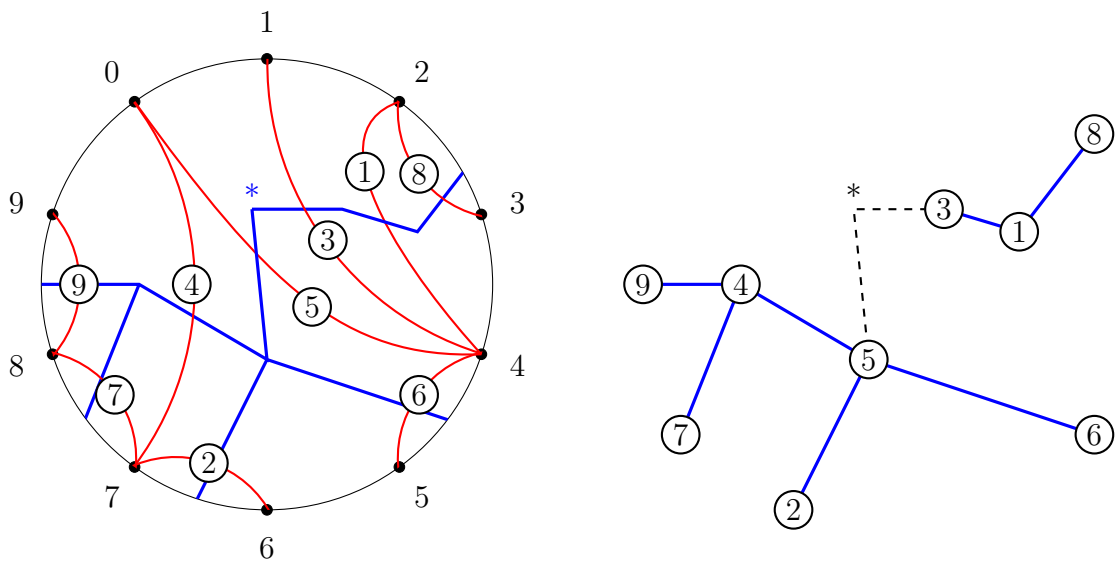

Goulden and Yong \cite{GY} (and earlier \cite{M}) showed that rooted trees with $n$ labeled edges are in bijection with factorizations of the $n+1$ cycle $(012\cdots n)$ as a product of $n$ transpositions $\tau_i=(a_i,b_i)$. The bijection is given by drawing chords (labeled with $i$) between points $a_i$ and $b_i$ in a circle as indicated in Figure \ref{fig: Goulden-Yong}. The dual graph of these chords is the tree. The root is the region adjacent to the arc from $0$ to $1$. Brady and Watt \cite{BW} showed that such factorizations of the $n+1$ cycle are in bijection with exceptional sequences where $(\tau_1,\cdots,\tau_n)=((a_1,b_1),\cdots,(a_n,b_n))$ corresponds to the exceptional sequence $(M_{a_1,b_1-1},\cdots,M_{a_n,b_{n}-1})$ where the indices are taken modulo $n+1$.

It is easy to see why the bijection between exceptional sequences and rooted labeled forests (or trees), as illustrated in Figure \ref{fig: Goulden-Yong} is the same as the bijection given in Theorem \ref{theorem A1}. Any chords in the construction, say $14$, separates the circle in two parts. One part contains the integers in the interval $[1,4]$. The other part contains the arc $01$ and therefore the root of the tree. So, any other chord $(a,b)$ will be further away from the root if and only if $[a,b]\subsetneq [1,4]$, or, equivalently, $[a,b)\subsetneq [1,4)$. Therefore, the rooted forest is the Hasse diagram of the set of corresponding intervals $[a_i,b_i)$ which are the supports of the objects in the exceptional sequence.

The correspondence between exceptional sequences and factorizations of the Coxeter element of the Weyl group (the $n+1$ cycle in the symmetric group for $A_n$) has been generalized to any quiver in \cite{Ig-Sh}. The correspondence with chord diagrams is also known for any orientation of $A_n$ \cite{GIMO}. See also \cite{IT} for the relation to noncrossing partitions.
%

\subsection{Generating function for exceptional sequences}

The following theorem and its proof come from \cite{GesselSeo}. (See also \cite{Sokal}.) This is a special case of a result from \cite{ER}.

\begin{thm}\label{thm: counting forests}
For each $n\ge1$, let $P_n(a,b,c)$ be the three variable polynomial for rooted labeled forest given by $P_n(a,b,c)=\sum a^pb^qc^r$ where the sum is over all rooted labeled forests, $r$ is the number of roots of the forest, $p$ is the number of nodes $v_i$ whose parent $v_j$ has label $j>i$, $q$ is the number of nodes $v_i$ with parent $v_j$ so that $j<i$. Then
\[
	P_n(a,b,c)=c(a+(n-1)b+c)(2a+(n-2)b+c)\cdots((n-1)a+b+c).
\]
\end{thm}

\begin{proof} We associate a weight to each node of a forest. The weight will be $a,b,c$ depending on which type of node it is. Then the monomial $a^pb^qc^r$ for a forest is the product of the weights of its nodes.

The proof is by induction on $n$. For $n=1$, $P_1(a,b,c)=c$ since there is only one forest with one node which has weight $c$.

For $n\ge2$, we write $P_n(a,b,c)$ as the sum of two generating function
\[
	P_n(a,b,c)=P_n(a,b,c)_1+P_n(a,b,c)_2
\]
the first is for forest in which $v_1$ is a root. We remove the root $v_1$, decrease the labels of the other nodes by 1 and keep the weights of the children of $v_1$ to be $b$. This gives a rooted labeled forest with $n-1$ nodes in which the roots have weight either $b$ or $c$. Since a node with weight $c$ was removed we have:
\[
	P_n(a,b,c)_1=cP_{n-1}(a,b,b+c).
\]
For the remaining forests, we remove the node $v_1$ and make the children of $v_1$ roots, but keeping the weight $b$. We obtain a forest with roots of weight either $b$ or $c$. To reassemble the original forest, we need to choose one node to be the parent of $v_1$, we need to make the root of the tree in which this node lives to have weight $c$. The other $b$ weighted roots become the children of $v_1$. So,
\[
	P_n(a,b,c)_2=a(n-1)\frac{c}{b+c}P_{n-1}(a,b,b+c).
\]
A straightforward calculation proves the result.
\end{proof}

Let $P_{A_n}(a,b,c)=\sum \lambda_{pqr}a^pb^qc^r$ where $\lambda_{pqr}$ is the number of complete exceptional sequences for linear $A_n$ having 
\begin{enumerate}
\item $p$ objects $E_i$ which are relatively projective but not relatively injective,
\item $q$ objects which are relatively injective but not relatively projective
\item $r$ objects which are both relatively projective and relatively injective.
\end{enumerate}

\begin{thm}\label{thm: main theorem}(Theorem \ref{cor C})
\[
	P_{A_n}(a,b,c)=P_n(a,b,c)=c\prod_{j=1}^{n-1}(ja+(n-j)b+c).
\]
\end{thm}

\begin{proof}
This follows from Theorem \ref{thm: counting forests} and the properties of the bijection between complete exceptional sequences and rooted labeled forests given by Theorem \ref{theorem A2}.
\end{proof}

For example, when $n=3$ this is:
\[
	P_{A_3}(a,b,c)=P_3(a,b,c)=c(a+2b+c)(2a+b+c)=2a^2c+2b^2c+5abc+3ac^2+3bc^2+c^3.
\]

\begin{rem}\label{rem: result is only for straight orientation}
Theorem \ref{thm: main theorem} implies that there is only one exceptional sequence in which all objects are both relatively projective and relatively injective. (The $c^3$ term in $P_{A_3}$ above.) This is never true for other orientations of $A_n$ since such an exceptional sequence is given by a sequence of simple modules in which the support of each object is a source in the complement of the supports of the ones on its left and there are at least two such sequences for any nonlinear $A_n$. For example, for $A_3: 1\to 2\leftarrow 3$, there are two exceptional sequences with all terms relatively projective and relatively injective. This gives the term $2c^3$ in its three variable generating function:
\[
	P_{1\to 2\leftarrow 3}(a,b,c)=2a^2b+2b^2c+4abc+4ac^2+2bc^2+2c^3
\]
So, Theorem \ref{thm: main theorem} hold only for $A_n$ with straight orientation.
\end{rem}

\section{Braid group action on rooted labeled forests}\label{sec 3}\label{braid actionnnnnn}

It is well-known that the braid group acts transitively on the set of complete exceptional sequences for any Dynkin quiver, in particular, for type $A_n$ \cite{Crawley-Boevey}, \cite{RingelExcSeq}. In this section we construct an natural action of the braid group on rooted labeled forests and we show, in Theorem \ref{thm: action of sigma-i corresponds}, that this action corresponds to the action on complete exceptional sequences under the bijection given in the previous section. We also show that this action of the braid group is different from the one given in \cite{GG} since we are using a different bijection between the two sets.

The \emph{braid group} $B_n$ on $n$ strands is given by generators and relations as follows. The generators are $\sigma_i$ for $1\le i<n$ with two relations:
\begin{enumerate}
\item $\sigma_i,\sigma_j$ commute if $|j-i|\ge2$
\item $\sigma_i\sigma_{i+1}\sigma_i=\sigma_{i+1}\sigma_i\sigma_{i+1}$.
\end{enumerate}

Given a rooted labeled forest $F$, two node $v_i,v_j$ are said to be \emph{close} if one of the following holds.
\begin{enumerate}
\item $v_i,v_j$ are roots of distinct components of $F$.
\item $v_i,v_j$ are sibling, i.e., they have the same parent.
\item One of $v_i,v_j$ is the parent of the other.
\end{enumerate}
All remaining cases are not close.
 
To simplify notation, we add a root $v_0$, which we call the \emph{master root}, at the top of the forest giving a rooted tree $F_+$. Then Case 1 is a special case of Case 2: Roots in $F$ are sibling in $F_+$.

\subsection{Braid group action}\label{def: braid group action on forests}
The braid group $B_n$ acts on the set of rooted labeled forests on $n$ as follows. We give the action of $\sigma_i$ for $1\le i<n$. The partial ordering on the set $[n]=\{1,2,\cdots,n\}$ with Hasse diagram $F$ and the one with Hasse diagram $\sigma_i(F)$ are the same for nodes other than $v_i,v_{i+1}$. We keep the same notation for these other nodes in $\sigma_i(F)$. But, we denote by $v_i',v_{i+1}'$ the new nodes of $\sigma_i(F)$ with labels $i,i+1$. These are given as follows. See Figure \ref{fig: braid example}.

\underline{Case 0}: If nodes $v_i,v_j$ are not close then $\sigma_i(F)$ is given by switching the labels $i,i+1$, i.e., $\sigma_i(F)= F$ with nodes relabeled: $v_i'=v_{i+1},v_{i+1}'=v_i$.

\underline{Case 1}: Suppose $v_i$ is the parent of $v_{i+1}$, $v_k$ is the parent of $v_i$ in $F_+$, $X$ is the set of other children of $v_i$ and $Y$ is the set of children of $v_{i+1}$. Then $\sigma_i(F)_+$ is given by removing node $v_{i+1}$, relabeling $v_i$ to $v_{i+1}'$, then adding a new node $v_i'$ and making it a child of  $v_{i+1}'$ and a parent of all elements of $X$. Thus $Y\cup\{v_i'\}$ is the set of children of the new node $v_{i+1}'$ and $v_k$ is the parent of $v_{i+1}'$ in $\sigma_i(F)_+$. 

\underline{Case 2}: Suppose $v_i$ is a child of $v_{i+1}$, $v_k$ is the parent of $v_{i+1}$ in $F_+$, $Y$ is the set of other children of $v_{i+1}$ and $X$ is the set of children  of $v_i$. Then $\sigma_i(F)_+$ is given by removing node $v_{i+1}$, relabeling $v_i$ to $v_{i+1}'$, then adding a new node $v_i'$ which will be the parent of the elements of $Y$. The nodes $v_i',v_{i+1}'$ will be sibling in $\sigma_i(F)_+$, with parent $v_k$. 

\underline{Case 3}: Suppose $v_i,v_{i+1}$ are sibling with parent $v_k$ in $F_+$ and $Y,X$ are the sets of children of $v_i,v_{i+1}$, resp.  Then $\sigma_i(F)_+$ is given by removing node $v_{i+1}$, relabeling $v_i$ to $v_{i+1}'$, then adding a new node $v_i'$ with parent $v_k$ and children $v_{i+1}'$ and the elements of $X$.

Note that $\sigma_i$ takes examples from Cases 1,2,3 to examples from Cases 2,3,1 respectively. Furthermore, $\sigma_i^3$ is the identity map on these three cases. See Figure \ref{fig: braid example}

\begin{figure}[h]
\[
\xymatrixrowsep{12pt}\xymatrixcolsep{10pt}
\xymatrix{
v_k\ar@{-}[d] &&&& v_k\ar@{-}[d] &  & && v_k\ar@{-}[d]\ar@{-}[dd]&  &  \\ 
v_i\ar@{-}[dd]\ar@{-}[dr]  & Z\ar@{-}[ul] &&& v_{i+1}\ar@{-}[ddr] &Z\ar@{-}[ul]&&&  &  Z\ar@{-}[ul] &\\ 
 & v_{i+1}\ar@{-}[d] &   \ar@{=>}[r]^{\sigma_i}  & & v_i\ar@{-}[d]\ar@{-}[u] & &  \ar@{=>}[r]^{\sigma_i}  && v_{i+1}\ar@{-}[d] &  v_i\ar@{-}[uul]\ar@{-}[d] & \ar@{=>}[r]^{\sigma_i} & & (\text{Case 1}) \\ 
	X&Y&&&X  &Y  & && X & Y\\
	\text{Case 1} & &&&\text{Case 2} & &&&\text{Case 3} &&& 
	}
\]
\caption{Cases 1,2,3 of the braid move $\sigma_i$ are indicated. $v_k$ denotes the smallest node in $F_+$ which is above both $v_i$ and $v_{i+1}$ and $Z$ denotes the set of other children of $v_k$. $\sigma_i^3$ is the identity in all three cases.}\label{fig: braid example}
\end{figure}
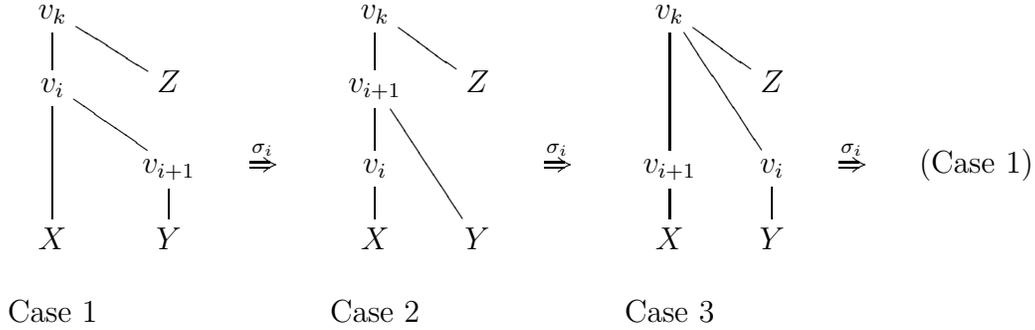

\begin{thm}\label{thm: action of sigma-i corresponds}
The action of $\sigma_i$ on the set of rooted labeled forests described above corresponds to the action of $\sigma_i$ on exceptional sequences under the bijection given in the previous section.
\end{thm}

\begin{cor}
The action of $\sigma_i$ on the set of rooted labeled forests satisfies the braid relations and therefore gives an action of the braid group.\qed
\end{cor}

The following property of $\sigma_i$, which follows easily from its description, will be needed in the next section.

\begin{prop}\label{root prop}
The vertex $v_{i+1}$ is a root of the forest $F$ if and only if $v_i'$ is a root of $\sigma_iF$.\qed
\end{prop}

\subsection{Proof of Theorem \ref{thm: action of sigma-i corresponds}}

\begin{lem}\label{lem: Ei and Ei+1} Let $F$ be a rooted labeled forest with associated exceptional sequence $E_\ast=(E_1,\cdots,E_n)$. Then for any $i<n$ we have the following.
\begin{enumerate}
\item $E_i$ is a submodule of $E_{i+1}$ if and only if $v_i$ is a child of $v_{i+1}$ in $F$.
\item $E_{i+1}$ is a quotient of $E_i$ if and only if $v_{i+1}$ is a child of $v_i$.
\item $E_i,E_{i+1}$ have consecutive supports if and only if $v_i,v_{i+1}$ are sibling in $F_+$.
\item $E_i,E_{i+1}$ are $\Hom$ and $\Ext$ orthogonal (in both directions) if and only if none of the above hold, i.e., if and only if $v_i,v_{i+1}$ are not close.
\end{enumerate}
\end{lem}

\begin{proof}
If $E_i\subset E_{i+1}$ then, by Lemma \ref{lem1: Ei submodule of Ej}, $v_i$ must be a child of $v_{i+1}$ since there are no $k$ with $i<k<i+1$. If $E_{i+1}$ is a quotient of $E_i$ then, by Lemma \ref{lem2: Ei quotient of Ej}, $v_{i+1}$ is a child of $v_i$ since there is no $k$ between $i$ and $i+1$. The converses are given by Proposition \ref{prop: properties of Ei} (a) and (b) which follow from Proposition \ref{prop: children of the root}. This proves (1) and (2). 

If $E_i,E_{i+1}$ have consecutive supports, Lemma \ref{lem3: Ei, Ej consecutive} implies $v_i,v_{i+1}$ must be sibling since we must have $i=k<\ell=j=i+1$. The converse is given by Proposition \ref{prop: properties of Ei} (c) and (d) and by Proposition \ref{prop: F disconnected case} when $v_i,v_{i+1}$ are both roots. This proves (3). All other cases are in (4).
\end{proof}

\begin{proof}[Proof of Theorem \ref{thm: action of sigma-i corresponds}]
The braid move $\sigma_i$ acts on $E_\ast$ in all cases by moving $E_i$ to position $i+1$ and inserting a uniquely determined new object $E_{i}'$ in position $i$ to give a new exceptional sequence
\[
	\sigma_i(E_\ast)=(E_1,\cdots,E_{i-1},E_{i}',E_i,E_{i+2},\cdots,E_n).
\]

\underline{Case a}: If $E_i,E_{i+1}$ are $\Hom$ and $\Ext$ orthogonal then they commute and $E_{i}'=E_{i+1}$.

\underline{Case b}: Suppose that $E_i$ is a submodule of $E_{i+1}$. Then we know that $E_i'$ is the quotient module $E_{i+1}/E_i$. By Lemma \ref{lem: Ei and Ei+1}(1), $v_i$ is a child of $v_{i+1}$. Let $Y$ denote the other children of $v_{i+1}$. These have disjoint supports whose union is the support of $E_{i+1}/E_i$ minus one point. Therefore, $Y$ is the set of children of $v_i'$ the node in $\sigma_i(F)$ corresponding to $E_i'$. $E_i$ has not changed except for its label which is now $i+1$. So, it has the same set of children. This is Case 2 in subsection \ref{def: braid group action on forests} and $\sigma_i(F)$ is the Hasse diagram of $\sigma_i(E_\ast)$.

\underline{Case c}: Suppose that $E_{i+1}$ is a quotient of $E_i$, say $E_{i+1}=M_{ab}$ and $E_i=M_{ac}$. Then $E_i'=M_{b+1,c}$ is the kernel of the epimorphism $E_i\onto E_{i+1}$. By Lemma \ref{lem: Ei and Ei+1}(2), $v_{i+1}$ is a child of $v_i$. If $X$ is the set of the other children of $v_i$ in $F$ then the corresponding modules must have support in $[a,c]$ but disjoint from $[a,b]$. So, their supports are in $\supp E_i'=[b+1,c]$. So, $X$ is the set of children of the new $v_i'$ in $\sigma_i(F)$. Since $v_{i+1}$ in $F$ has been removed, its children $Y$ become children of $v_i$, the new $v_{i+1}'\in \sigma_i(F)$.

\underline{Case d}: Suppose that $v_i,v_{i+1}$ are sibling in $F_+$ with sets of children $Y,X$, resp. Then $E_i,E_{i+1}$ must have disjoint consecutive supports by Lemma \ref{lem: Ei and Ei+1}(3), say $\supp E_i=[a,b]$, $\supp E_{i+1}=[b+1,c]$. Since $E_i$ extends $E_{i+1}$, the new exceptional pair $(E_i',E_i)$ must be $(M_{ac},M_{ab})$. The children of the new $v_i'$ with module $M_{ac}$ are the new $v_{i+1}'$ with module $M_{ab}$ and $X$. The children of $v_{i+1}'$ with module $M_{ab}$ are $Y$.

This concludes the proof of Theorem \ref{thm: action of sigma-i corresponds}. See Figure \ref{fig: Braid example 2} for an example.
\end{proof}

Figure \ref{fig: Braid example 2} illustrates the sequence of five rooted labeled forests corresponding to the following five complete exceptional sequences for $A_{10}$:
 \[
 \begin{array}{rl}
&(M_{99},M_{1,10},M_{46},M_{16},M_{12},\ M_{79},M_{77},M_{11},M_{44},M_{55})\\
\xrightarrow{\sigma_3}&(M_{99},M_{1,10},M_{13},M_{46},M_{12},\ M_{79},M_{77},M_{11},M_{44},M_{55})\\
\xrightarrow{\sigma_2}&(M_{99},M_{4,10},M_{1,10},M_{46},M_{12},\ M_{79},M_{77},M_{11},M_{44},M_{55})\\
\xrightarrow{\sigma_5}&(M_{99},M_{4,10},M_{1,10},M_{46},M_{79},\ M_{12},M_{77},M_{11},M_{44},M_{55})\\
\xrightarrow{\sigma_4}&(M_{99},M_{4,10},M_{1,10},M_{49},M_{46},\ M_{12},M_{77},M_{11},M_{44},M_{55})
\end{array}
\]

\begin{figure}[ht]
\begin{center}
\begin{tikzpicture}[scale=.9]
\begin{scope}[xshift=0cm]
\draw[fill] (0,0) circle[radius=.7mm];
\draw[fill] (1,0) circle[radius=.7mm];
\draw[fill] (2,0) circle[radius=.7mm];
\draw[fill] (3,0) circle[radius=.7mm];
\draw[fill] (4,0) circle[radius=.7mm];
\draw[fill] (.5,1) circle[radius=.7mm];
\draw[fill] (2,1) circle[radius=.7mm];
\draw[fill] (3.5,1) circle[radius=.7mm];
\draw[fill] (3,2) circle[radius=.7mm];
\draw[fill] (2,3) circle[radius=.7mm];
\draw[thick] (0,0)--(0.5,1)--(1,0) (0.5,1)--(2,3)--(3,2)--(2,1)--(2,0) (3,2)--(3.5,1)--(3,0) (4,0)--(3.5,1);
\draw[white,fill] (0,0) circle[radius=3mm];
\draw (0,0) circle[radius=3mm];
\draw (0,0) node{1};
\draw[white,fill] (.5,1) circle[radius=3mm];
\draw (.5,1) circle[radius=3mm];
\draw (.5,1) node{6};
\draw[white,fill] (3.5,1) circle[radius=3mm];
\draw (3.5,1) circle[radius=3mm];
\draw (3.5,1) node{3};
\draw[white,fill] (3,2) circle[radius=3mm];
\draw (3,2) circle[radius=3mm];
\draw (3,2) node{4};
\draw[white,fill] (1,0) circle[radius=3mm];
\draw (1,0) circle[radius=3mm];
\draw (1,0) node{7};
\draw[white,fill] (2,0) circle[radius=3mm];
\draw (2,0) circle[radius=3mm];
\draw (2,0) node{8};
\draw[white,fill] (2,1) circle[radius=3mm];
\draw (2,1) circle[radius=3mm];
\draw (2,1) node{5};
\draw[white,fill] (2,3) circle[radius=3mm];
\draw (2,3) circle[radius=3mm];
\draw (2,3) node{2};
\draw[white,fill] (3,0) circle[radius=3mm];
\draw (3,0) circle[radius=3mm];
\draw (3,0) node{9};
\draw[white,fill] (4,0) circle[radius=3mm];
\draw (4,0) circle[radius=3mm];
\draw (4,0) node{10};
\end{scope}
\begin{scope}[xshift=4.5cm, yshift=1.5cm]
\draw[thick,->] (.1,0)--(.9,0);
\draw (.5,.3)node{$\sigma_3$};
\end{scope}
\begin{scope}[xshift=6cm]
\draw[fill] (0,0) circle[radius=.7mm];
\draw[fill] (1,0) circle[radius=.7mm];
\draw[fill] (2,0) circle[radius=.7mm];
\draw[fill] (3,0) circle[radius=.7mm];
\draw[fill] (4,0) circle[radius=.7mm];
\draw[fill] (.5,1) circle[radius=.7mm];
\draw[fill] (2,1) circle[radius=.7mm];
\draw[fill] (3.5,1) circle[radius=.7mm];
\draw[fill] (2,3) circle[radius=.7mm];
\draw[thick] (0,0)--(0.5,1)--(1,0) (0.5,1)--(2,3)--(2,1)--(2,0) (2,3)--(3.5,1)--(3,0) (4,0)--(3.5,1);
\draw[white,fill] (0,0) circle[radius=3mm];
\draw (0,0) circle[radius=3mm];
\draw (0,0) node{1};
\draw[white,fill] (.5,1) circle[radius=3mm];
\draw (.5,1) circle[radius=3mm];
\draw (.5,1) node{6};
\draw[white,fill] (3.5,1) circle[radius=3mm];
\draw (3.5,1) circle[radius=3mm];
\draw (3.5,1) node{4};
\draw[white,fill] (2,2) circle[radius=3mm];
\draw (2,2) circle[radius=3mm];
\draw (2,2) node{3};
\draw[white,fill] (1,0) circle[radius=3mm];
\draw (1,0) circle[radius=3mm];
\draw (1,0) node{7};
\draw[white,fill] (2,0) circle[radius=3mm];
\draw (2,0) circle[radius=3mm];
\draw (2,0) node{8};
\draw[white,fill] (2,1) circle[radius=3mm];
\draw (2,1) circle[radius=3mm];
\draw (2,1) node{5};
\draw[white,fill] (2,3) circle[radius=3mm];
\draw (2,3) circle[radius=3mm];
\draw (2,3) node{2};
\draw[white,fill] (3,0) circle[radius=3mm];
\draw (3,0) circle[radius=3mm];
\draw (3,0) node{9};
\draw[white,fill] (4,0) circle[radius=3mm];
\draw (4,0) circle[radius=3mm];
\draw (4,0) node{10};
\end{scope}
\begin{scope}[xshift=10.5cm, yshift=1.5cm]
\draw[thick,->] (.1,0)--(.9,0);
\draw (.5,.3)node{$\sigma_2$};
\end{scope}
\begin{scope}[xshift=12cm]
\draw[thick] (0,0)--(0.5,1)--(1,0) 
(0.5,1)--(2,3) 
(1.25,2)--(2.5,1)--(2,0) 
(2,3)--(3.5,1) 
(4,0)--(3.5,1)
(3,0)--(2.5,1);
\draw[white,fill] (0,0) circle[radius=3mm];
\draw (0,0) circle[radius=3mm];
\draw (0,0) node{1};
\draw[white,fill] (.5,1) circle[radius=3mm];
\draw (.5,1) circle[radius=3mm];
\draw (.5,1) node{6};
\draw[white,fill] (3.5,1) circle[radius=3mm];
\draw (3.5,1) circle[radius=3mm];
\draw (3.5,1) node{5};
\draw[white,fill] (1.25,2) circle[radius=3mm];
\draw (1.25,2) circle[radius=3mm];
\draw (1.25,2) node{2};
\draw[white,fill] (1,0) circle[radius=3mm];
\draw (1,0) circle[radius=3mm];
\draw (1,0) node{7};
\draw[white,fill] (2,0) circle[radius=3mm];
\draw (2,0) circle[radius=3mm];
\draw (2,0) node{9};
\draw[white,fill] (2.5,1) circle[radius=3mm];
\draw (2.5,1) circle[radius=3mm];
\draw (2.5,1) node{4};
\draw[white,fill] (2,3) circle[radius=3mm];
\draw (2,3) circle[radius=3mm];
\draw (2,3) node{3};
\draw[white,fill] (3,0) circle[radius=3mm];
\draw (3,0) circle[radius=3mm];
\draw (3,0) node{10};
\draw[white,fill] (4,0) circle[radius=3mm];
\draw (4,0) circle[radius=3mm];
\draw (4,0) node{8};
\end{scope}
\begin{scope}[xshift=.5cm, yshift=-4cm]
\draw[thick,->] (.1,0)--(.9,0);
\draw (.5,.3)node{$\sigma_5$};
\end{scope}
\begin{scope}[xshift=2.5cm,yshift=-5.5cm]
\draw[thick] (0,0)--(0.5,1)--(1,0) 
(0.5,1)--(2,3) 
(1.25,2)--(2.5,1)--(2,0) 
(2,3)--(3.5,1) 
(4,0)--(3.5,1)
(3,0)--(2.5,1);
\draw[white,fill] (0,0) circle[radius=3mm];
\draw (0,0) circle[radius=3mm];
\draw (0,0) node{1};
\draw[white,fill] (.5,1) circle[radius=3mm];
\draw (.5,1) circle[radius=3mm];
\draw (.5,1) node{5};
\draw[white,fill] (3.5,1) circle[radius=3mm];
\draw (3.5,1) circle[radius=3mm];
\draw (3.5,1) node{6};
\draw[white,fill] (1.25,2) circle[radius=3mm];
\draw (1.25,2) circle[radius=3mm];
\draw (1.25,2) node{2};
\draw[white,fill] (1,0) circle[radius=3mm];
\draw (1,0) circle[radius=3mm];
\draw (1,0) node{7};
\draw[white,fill] (2,0) circle[radius=3mm];
\draw (2,0) circle[radius=3mm];
\draw (2,0) node{9};
\draw[white,fill] (2.5,1) circle[radius=3mm];
\draw (2.5,1) circle[radius=3mm];
\draw (2.5,1) node{4};
\draw[white,fill] (2,3) circle[radius=3mm];
\draw (2,3) circle[radius=3mm];
\draw (2,3) node{3};
\draw[white,fill] (3,0) circle[radius=3mm];
\draw (3,0) circle[radius=3mm];
\draw (3,0) node{10};
\draw[white,fill] (4,0) circle[radius=3mm];
\draw (4,0) circle[radius=3mm];
\draw (4,0) node{8};
\end{scope}
\begin{scope}[xshift=7.5cm, yshift=-4cm]
\draw[thick,->] (.1,0)--(.9,0);
\draw (.5,.3)node{$\sigma_4$};
\end{scope}
\begin{scope}[xshift=9.5cm,yshift=-5.6cm]
\draw[thick] (0,0)--(1.25,2)--(1,0) 
(1.25,2)--(2.75,4) 
(1.25,2)--(2.5,1)--(2,0) 
(2.75,4)--(4,1)--(4,0)
(3,0)--(2.5,1);
\draw[white,fill] (0,0) circle[radius=3mm];
\draw (0,0) circle[radius=3mm];
\draw (0,0) node{1};
\draw[white,fill] (4,1) circle[radius=3mm];
\draw (4,1) circle[radius=3mm];
\draw (4,1) node{6};
\draw[white,fill] (1.25,2) circle[radius=3mm];
\draw (1.25,2) circle[radius=3mm];
\draw (1.25,2) node{4};
\draw[white,fill] (1,0) circle[radius=3mm];
\draw (1,0) circle[radius=3mm];
\draw (1,0) node{7};
\draw[white,fill] (2,0) circle[radius=3mm];
\draw (2,0) circle[radius=3mm];
\draw (2,0) node{9};
\draw[white,fill] (2.5,1) circle[radius=3mm];
\draw (2.5,1) circle[radius=3mm];
\draw (2.5,1) node{5};
\draw[white,fill] (2,3) circle[radius=3mm];
\draw (2,3) circle[radius=3mm];
\draw (2,3) node{2};
\draw[white,fill] (2.75,4) circle[radius=3mm];
\draw (2.75,4) circle[radius=3mm];
\draw (2.75,4) node{3};
\draw[white,fill] (3,0) circle[radius=3mm];
\draw (3,0) circle[radius=3mm];
\draw (3,0) node{10};
\draw[white,fill] (4,0) circle[radius=3mm];
\draw (4,0) circle[radius=3mm];
\draw (4,0) node{8};
\end{scope}
\end{tikzpicture}
\caption{Braid moves $\sigma_5,\sigma_3,\sigma_2,\sigma_4$ illustrated above are examples of Cases a,b,c,d in the proof of Theorem \ref{thm: action of sigma-i corresponds}.}
\label{fig: Braid example 2}
\end{center}
\end{figure}
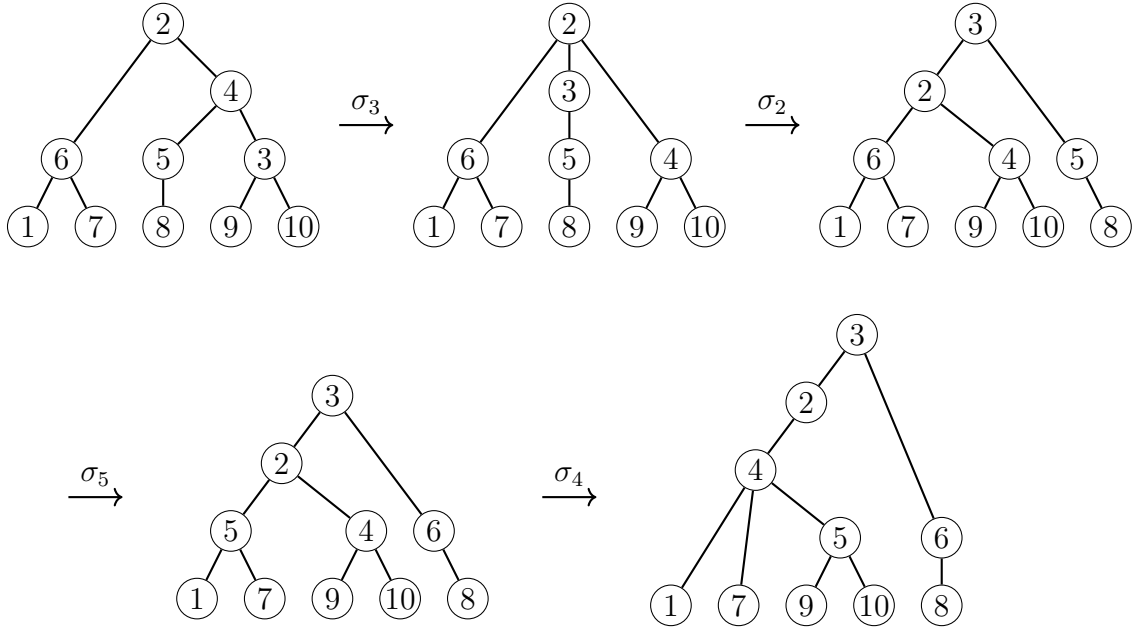

\subsection{Examples: $A_2$ and $A_3$}

For $n=2$, there are 3 rooted labeled forests and $\sigma_1$ permutes them in a 3-cycle:
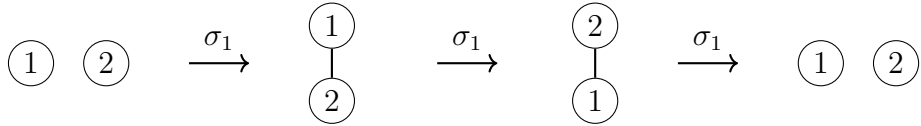
\begin{figure}[ht]
\begin{center}
\begin{tikzpicture}
\begin{scope}
\draw[white,fill] (0,0) circle[radius=3mm];
\draw (0,0) circle[radius=3mm];
\draw[white,fill] (1,0) circle[radius=3mm];
\draw (1,0) circle[radius=3mm];
\draw (0,0) node{1};
\draw (1,0) node{2};
\end{scope}
\draw[thick,->] (2.1,0)--(2.9,0);
\draw (2.5,.3)node{$\sigma_1$};
\begin{scope}[xshift=3.3cm]
\draw[thick,->] (2.1,0)--(2.9,0);
\draw (2.5,.3)node{$\sigma_1$};
\end{scope}
\begin{scope}[xshift=6.5cm]
\draw[thick,->] (2.1,0)--(2.9,0);
\draw (2.5,.3)node{$\sigma_1$};
\end{scope}
\begin{scope}[xshift=3.5cm]
\draw[thick] (0.5,.5)--(0.5,-.5);
\draw[white,fill] (0.5,.5) circle[radius=3mm];
\draw (0.5,.5) circle[radius=3mm];
\draw[white,fill] (0.5,-.5) circle[radius=3mm];
\draw (0.5,-.5) circle[radius=3mm];
\draw (0.5,.5) node{1};
\draw (.5,-.5) node{2};
\end{scope}
\begin{scope}[xshift=7cm]
\draw[thick] (0.5,.5)--(0.5,-.5);
\draw[white,fill] (0.5,.5) circle[radius=3mm];
\draw (0.5,.5) circle[radius=3mm];
\draw[white,fill] (0.5,-.5) circle[radius=3mm];
\draw (0.5,-.5) circle[radius=3mm];
\draw (0.5,.5) node{2};
\draw (.5,-.5) node{1};
\end{scope}
\begin{scope}[xshift=10.5cm]
\draw[white,fill] (0,0) circle[radius=3mm];
\draw (0,0) circle[radius=3mm];
\draw[white,fill] (1,0) circle[radius=3mm];
\draw (1,0) circle[radius=3mm];
\draw (0,0) node{1};
\draw (1,0) node{2};
\end{scope}
\end{tikzpicture}
\caption{The three rooted labeled forests on $n=2$ are cyclically permuted by $\sigma_1$.}
\label{fig: A2 example}
\end{center}
\end{figure}

For $n=3$ there are 16 rooted labeled forests. Figure \ref{fig: A3 example 1} illustrates 8 of them. Figure \ref{fig: A3 example AR} illustrates how the rooted labeled forest can be visualized when the exceptional sequence is embedded in the Auslander-Reiten quiver of the path algebra.
\begin{figure}[ht]
\begin{center}
\begin{tikzpicture}
\begin{scope}[yshift=2.5cm] 
\begin{scope}[xshift=-3.8cm]
\draw[thick,->] (2.1,0)--(2.9,0);
\draw (2.5,.3)node{$\sigma_1$};
\end{scope}
\begin{scope}[xshift=-4cm]
\draw[white,fill] (1.5,0) circle[radius=3mm];
\draw (1.5,0) circle[radius=3mm];
\draw[thick] (0.5,.5)--(0.5,-.5);
\draw[white,fill] (0.5,.5) circle[radius=3mm];
\draw (0.5,.5) circle[radius=3mm];
\draw[white,fill] (0.5,-.5) circle[radius=3mm];
\draw (0.5,-.5) circle[radius=3mm];
\draw (0.5,.5) node{2};
\draw (.5,-.5) node{1};
\draw (1.5,0) node{3};
\end{scope}
\begin{scope} 
\draw[white,fill] (0,0) circle[radius=3mm];
\draw (0,0) circle[radius=3mm];
\draw[white,fill] (1,0) circle[radius=3mm];
\draw (1,0) circle[radius=3mm];
\draw[white,fill] (2,0) circle[radius=3mm];
\draw (2,0) circle[radius=3mm];
\draw (0,0) node{1};
\draw (1,0) node{2};
\draw (2,0) node{3};
\end{scope}
\draw[thick,->] (3.1,0)--(3.9,0);
\draw (3.5,.3)node{$\sigma_2$};
\begin{scope}[xshift=5.3cm]
\draw[thick,->] (2.1,0)--(2.9,0);
\draw (2.5,.3)node{$\sigma_1$};
\end{scope}
\begin{scope}[xshift=4.5cm]
\draw[white,fill] (1.5,0) circle[radius=3mm];
\draw (1.5,0) circle[radius=3mm];
\draw[thick] (0.5,.5)--(0.5,-.5);
\draw[white,fill] (0.5,.5) circle[radius=3mm];
\draw (0.5,.5) circle[radius=3mm];
\draw[white,fill] (0.5,-.5) circle[radius=3mm];
\draw (0.5,-.5) circle[radius=3mm];
\draw (0.5,.5) node{2};
\draw (.5,-.5) node{3};
\draw (1.5,0) node{1};
\end{scope}
\begin{scope}[xshift=9cm]
\draw[thick] (0.5,-.5)--(1,.5)--(1.5,-.5);
\draw[white,fill] (1,.5) circle[radius=3mm];
\draw (1,.5) circle[radius=3mm];
\draw[white,fill] (1.5,-.5) circle[radius=3mm];
\draw (1.5,-.5) circle[radius=3mm];
\draw[white,fill] (0.5,-.5) circle[radius=3mm];
\draw (0.5,-.5) circle[radius=3mm];
\draw (1,.5) node{1};
\draw (.5,-.5) node{2};
\draw (1.5,-.5) node{3};
\end{scope}
\end{scope} 
\begin{scope}[xshift=-2.5cm] 
\begin{scope} 
\draw[thick] (0,1)--(0,-1);
\draw[white,fill] (0,0) circle[radius=3mm];
\draw (0,0) circle[radius=3mm];
\draw[white,fill] (0,1) circle[radius=3mm];
\draw (0,1) circle[radius=3mm];
\draw[white,fill] (0,-1) circle[radius=3mm];
\draw (0,-1) circle[radius=3mm];
\draw (0,0) node{2};
\draw (0,1) node{1};
\draw (0,-1) node{3};
\end{scope}
\begin{scope}[xshift=-2cm]
\draw[thick,->] (.1,0)--(.9,0);
\draw (.5,.3)node{$\sigma_2$};
\end{scope}
\begin{scope}[xshift=-2cm]
\draw[thick,->] (3.1,0)--(3.9,0);
\draw (3.5,.3)node{$\sigma_1$};
\end{scope}
\begin{scope}[xshift=2.5cm]
\draw[thick] (0.5,-.5)--(1,.5)--(1.5,-.5);
\draw[white,fill] (1,.5) circle[radius=3mm];
\draw (1,.5) circle[radius=3mm];
\draw[white,fill] (1.5,-.5) circle[radius=3mm];
\draw (1.5,-.5) circle[radius=3mm];
\draw[white,fill] (0.5,-.5) circle[radius=3mm];
\draw (0.5,-.5) circle[radius=3mm];
\draw (1,.5) node{2};
\draw (.5,-.5) node{1};
\draw (1.5,-.5) node{3};
\end{scope}
\begin{scope}[xshift=2cm]
\draw[thick,->] (3.1,0)--(3.9,0);
\draw (3.5,.3)node{$\sigma_2$};
\end{scope}
\begin{scope}[xshift=7cm]
\draw[thick] (0,1)--(0,-1);
\draw[white,fill] (0,0) circle[radius=3mm];
\draw (0,0) circle[radius=3mm];
\draw[white,fill] (0,1) circle[radius=3mm];
\draw (0,1) circle[radius=3mm];
\draw[white,fill] (0,-1) circle[radius=3mm];
\draw (0,-1) circle[radius=3mm];
\draw (0,1) node{3};
\draw (0,0) node{2};
\draw (0,-1) node{1};
\end{scope}
\begin{scope}[xshift=5cm]
\draw[thick,->] (3.1,0)--(3.9,0);
\draw (3.5,.3)node{$\sigma_1$};
\end{scope}
\begin{scope}[xshift=9.5cm]
\draw[thick] (0.5,-.5)--(1,.5)--(1.5,-.5);
\draw[white,fill] (1,.5) circle[radius=3mm];
\draw (1,.5) circle[radius=3mm];
\draw[white,fill] (1.5,-.5) circle[radius=3mm];
\draw (1.5,-.5) circle[radius=3mm];
\draw[white,fill] (0.5,-.5) circle[radius=3mm];
\draw (0.5,-.5) circle[radius=3mm];
\draw (1,.5) node{3};
\draw (.5,-.5) node{1};
\draw (1.5,-.5) node{2};
\end{scope}
\begin{scope}[xshift=11.5cm]
\draw[thick,->] (.1,0)--(.9,0);
\draw (.5,.3)node{$\sigma_2$};
\end{scope}
\begin{scope}[xshift=13cm]
\draw[white,fill] (1.5,0) circle[radius=3mm];
\draw (1.5,0) circle[radius=3mm];
\draw[thick] (0.5,.5)--(0.5,-.5);
\draw[white,fill] (0.5,.5) circle[radius=3mm];
\draw (0.5,.5) circle[radius=3mm];
\draw[white,fill] (0.5,-.5) circle[radius=3mm];
\draw (0.5,-.5) circle[radius=3mm];
\draw (0.5,.5) node{2};
\draw (.5,-.5) node{1};
\draw (1.5,0) node{3};
\end{scope}
\end{scope} 
\end{tikzpicture}
\caption{The action of $\sigma_1$ and $\sigma_2$ is illustrated in the case $n=3$.}
\label{fig: A3 example 1}
\end{center}
\end{figure}
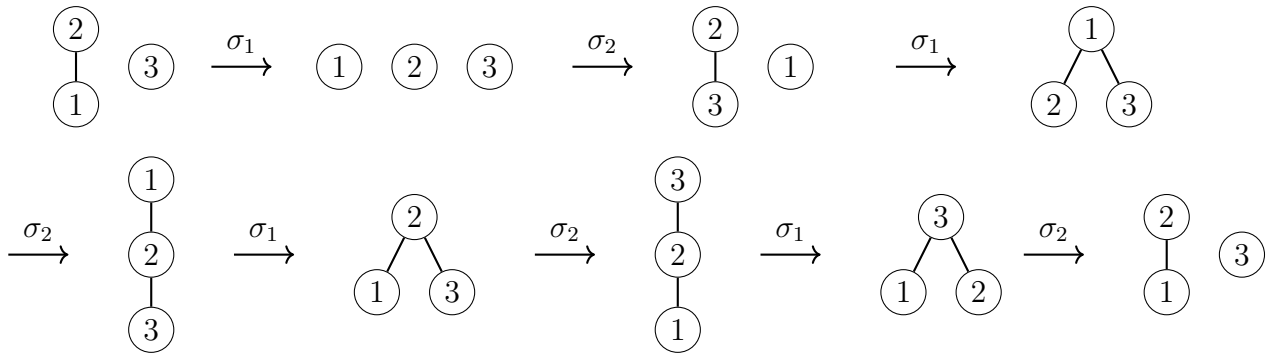

\begin{figure}[ht]
\begin{center}
\begin{tikzpicture}
\begin{scope}[xshift=0cm]
\draw[fill] (0,0) circle[radius=.7mm];
\draw[fill] (1.2,0) circle[radius=.7mm];
\draw[fill] (-1.2,0) circle[radius=.7mm];
\draw[fill] (-.6,.9) circle[radius=.7mm];
\draw[fill] (.6,.9) circle[radius=.7mm];
\draw[fill] (0,1.8) circle[radius=.7mm];
\draw[thick] (-1.2,0)--(0,1.8);
\draw[white,fill] (-1.2,0) circle[radius=3mm];
\draw (-1.2,0) circle[radius=3mm];
\draw[white,fill] (-.6,.9) circle[radius=3mm];
\draw (-.6,.9) circle[radius=3mm];
\draw[white,fill] (0,1.8) circle[radius=3mm];
\draw (0,1.8) circle[radius=3mm];
\draw (-1.2,0) node{1};
\draw (-.6,.9) node{2};
\draw (0,1.8) node{3};
\end{scope}
\begin{scope}[xshift=2cm, yshift=1cm]
\draw[thick,->] (.1,0)--(.9,0);
\draw (.5,.3)node{$\sigma_1$};
\end{scope}
\begin{scope}[xshift=5cm]
\draw[fill] (0,0) circle[radius=.7mm];
\draw[fill] (1.2,0) circle[radius=.7mm];
\draw[fill] (-1.2,0) circle[radius=.7mm];
\draw[fill] (-.6,.9) circle[radius=.7mm];
\draw[fill] (.6,.9) circle[radius=.7mm];
\draw[fill] (0,1.8) circle[radius=.7mm];
\draw[thick] (-1.2,0)--(0,1.8)--(0,0);
\draw[white,fill] (-1.2,0) circle[radius=3mm];
\draw (-1.2,0) circle[radius=3mm];
\draw[white,fill] (0,0) circle[radius=3mm];
\draw (0,0) circle[radius=3mm];
\draw[white,fill] (0,1.8) circle[radius=3mm];
\draw (0,1.8) circle[radius=3mm];
\draw (-1.2,0) node{2};
\draw (0,0) node{1};
\draw (0,1.8) node{3};
\end{scope}
\begin{scope}[xshift=7cm, yshift=1cm]
\draw[thick,->] (.1,0)--(.9,0);
\draw (.5,.3)node{$\sigma_2$};
\end{scope}

\begin{scope}[xshift=10cm]
\draw[fill] (0,0) circle[radius=.7mm];
\draw[fill] (1.2,0) circle[radius=.7mm];
\draw[fill] (-1.2,0) circle[radius=.7mm];
\draw[fill] (-.6,.9) circle[radius=.7mm];
\draw[fill] (.6,.9) circle[radius=.7mm];
\draw[fill] (0,1.8) circle[radius=.7mm];
\draw[thick] (0,0)--(.6,.9);
\draw[white,fill] (-1.2,0) circle[radius=3mm];
\draw (-1.2,0) circle[radius=3mm];
\draw[white,fill] (0,0) circle[radius=3mm];
\draw (0,0) circle[radius=3mm];
\draw[white,fill] (.6,.9) circle[radius=3mm];
\draw (.6,.9) circle[radius=3mm];
\draw (-1.2,0) node{3};
\draw (.6,.9) node{2};
\draw (0,0) node{1};
\end{scope}
\end{tikzpicture}
\caption{The last three forests in Figure \ref{fig: A3 example 1} drawn in the Auslander-Reiten sequence of $A_3$ following Remark \ref{rem: Hasse diagram in AR quiver}.}
\label{fig: A3 example AR}
\end{center}
\end{figure}
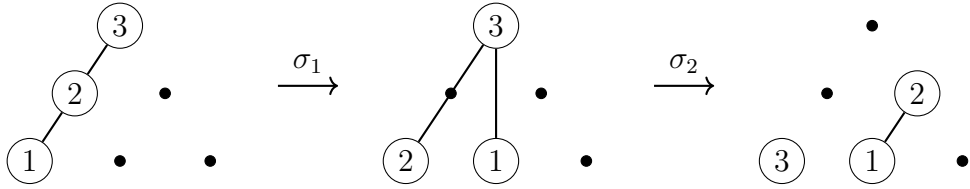

\section{Parking functions}\label{sec: parking functions}\label{parkingfunctions}

There is another action of the braid group on rooted labeled forests coming from the action on ``parking functions'' given in \cite{GG}. This action is difficult to describe, but we give an example to demonstrate that it is not the same as our action which is given in Figure \ref{fig: braid example}. First, we review the definitions and known results about parking functions.

A \emph{parking function} on $n$ is a function $f:[n]\to[n]$, where $[n]=\{1,2,\cdots,n\}$ so that $f^{-1}[k]$ has at least $k$ elements for all positive $k\le n$. This is equivalent to the condition that there are at most $p$ elements $i\in[n]$ with $f(i)\ge n-p+1$.

The name comes from an interpretation using $n$ cars parking in $n$ spaces numbered $1$ through $n$. Assume that Car $i$ has preferred parking space $a_i$. If that space is not available, Car $i$ will park in the next available space. If $f(i)=a_i$ is a parking function, each car $i$ will find a space $b_i\ge a_i$ and $g(i)=b_i$ will give a bijection $g:[n]\to [n]$.

As an example, suppose $n=4$ and consider the parking function $(1,1,2,2)$. Assume the cars come in reverse order. Then Car 4 will park in its preferred spot $b_4=2$, Car 3 will go to the next spot $b_3=3$, Car 2 will park in its favorite spot $b_2=1$. Finally, Car 1 will park in space $b_1=4$. The result is an exceptional sequence:
\[
 (M_{14}, M_{11}, M_{23}, M_{22}).
\]

\subsection{Parking functions and exceptional sequences} We review the well-known bijection between parking functions and exceptional sequences for linear $A_n$.

\begin{lem}\label{prop: nondecreasing parking function works}
Let $(a_1,a_2,\cdots,a_n)$ be a nondecreasing parking function on $n$. Let $b_1,\cdots,b_n$ be the permutation of $n$ given by the locations of the cars with preferred positions $a_i$ assuming they park in reverse order as outlined above. Then $M_{a_1,b_1},\cdots,M_{a_n,b_n}$ is an exceptional sequence for linear $A_n$.
\end{lem}

\begin{proof}
For $i<j$ we have $\Ext(M_{a_j,b_j},M_{a_i,b_i})=0$ since $a_i\le a_j$. It remains to show
\[
	\Hom(M_{a_j,b_j},M_{a_i,b_i})=0
\]
for $i<j$. When $a_i=a_j$, we have $b_i>b_j$ since Car $j$ has parked in space $b_j$ after finding spaces $a_j$ to $b_j-1$ occupied. Therefore, Car $i$ will find that spaces $a_i$ to $b_j$ are taken and must park in $b_i>b_j$. Then $M_{a_j,b_j}$ is a quotient of $M_{a_i,b_i}$ and $\Hom(M_{a_j,b_j},M_{a_i,b_i})=0$.

When $a_i<a_j$, there is no morphism $M_{a_j,b_j}\to M_{a_i,b_i}$ since $b_i$ is not in the closed interval $[a_j,b_j]$ since those spaces are occupied by the time Car $i$ parks. Thus either 
\begin{enumerate}
\item $b_i<a_j$, making the supports of the two modules disjoint or
\item $b_i>b_j$ in which case $M_{a_j,b_j}$ is a subquotient of $M_{a_i,b_i}$ and there is no morphism.
\end{enumerate}
Therefore $M_{a_1,b_1},\cdots,M_{a_n,b_n}$ is an exceptional sequence.
\end{proof}

\begin{lem}\label{lem: permutation on tops}
Let $(E_1,\cdots,E_n)$ be an exceptional sequence with sequence of tops $(a_1,\cdots,a_n)$. Then, for any $i$, there is another exceptional sequence with $a_i, a_{i+1}$
switched in the sequence of tops.
\end{lem}

\begin{proof}
If $a_i=a_{i+1}$ there is nothing to do. 
Otherwise, suppose $a_i<a_{i+1}$.
Then there is a unique object $E_{i+1}'$ so that $(E_1,\cdots,E_{i+1}',E_i,E_{i+2},\cdots,E_n)$ is an exceptional sequence. But, the dimension vector of $E_{i+1}'$ is congruent to $\undim E_{i+1}$ modulo $\undim E_i$. So, the top of $E_{i+1}'$ is equal to that of $E_{i+1}$, i.e,, the tops $a_i,a_{i+1}$ have switched.
The case $a_i>a_{i+1}$ is similar. 
\end{proof}

\begin{thm}\cite{GG}\label{thm A with proof}
The construction above gives a unique exceptional sequence having any given parking function as sequence of tops.
\end{thm}

\begin{proof}
By Lemma \ref{prop: nondecreasing parking function works}, the construction works for any nondecreasing parking function. By Lemma \ref{lem: permutation on tops}, we can permute the parking function. Since the two sets have the same cardinality $(n+1)^{n-1}$, this construction gives a bijection.
\end{proof}

\subsection{Parking functions and rooted labeled forests}\label{ss: comparison of braid group actions}

There is another well-known bijection between parking functions and rooted labeled forests using Pr\"ufer codes which we will review briefly here. We will use this bijection to compare the braid group action on parking functions given in \cite{GG} (using Theorem \ref{thm A with proof}) to the one we give in the previous section using rooted labeled forests.

We take as an example, the parking function $(1,1,1,1)$ on $A_4$ and the action of $\sigma_1$ then $\sigma_2$ then $\sigma_3$ using the top function (Theorem \ref{thm A with proof}):
\[
	(1,1,1,1)\xrightarrow{\sigma_1} (4,1,1,1)
	\xrightarrow{\sigma_2} (4,3,1,1)
	\xrightarrow{\sigma_3} (4,3,2,1).
\]
The \emph{Pr\"ufer code} corresponding to a parking function $(a_1,a_2,\cdots,a_n)$ is given by
\[
	(a_2-a_1,a_3-a_2,\cdots,a_n-a_{n-1})
\]
where these numbers are taken modulo $n+1$ which is $4+1=5$ in our case. Thus, the Pr\"ufer codes corresponding to our sequence of parking functions are:
\[
	(0,0,0)\xrightarrow{\sigma_1} (2,0,0)
	\xrightarrow{\sigma_2} (4,3,0)
	\xrightarrow{\sigma_3} (4,4,4).
\]
For example, the parking function $(4,3,1,1)$ has Pr\"ufer code $(-1,-2,0)=(4,3,0)$. 

The bijection between Pr\"ufer codes and rooted labeled forests, given below, gives the corresponding sequence of forests: %
\begin{equation}\label{eq: braid action on forests from GG action}
\begin{tikzpicture}[scale=.7]
\begin{scope}
\coordinate (A1) at (0,0);
\coordinate (A2) at (1,0);
\coordinate (A3) at (2,0);
\coordinate (A4) at (3,0);
\foreach \x/\y in {A1/1,A2/2,A3/3,A4/4}
\draw (\x) circle[radius=3mm] node{\y};
\end{scope}
\begin{scope}[xshift=4cm]
\draw (.5,0.15)node{$\xrightarrow{\sigma_1}$};
\coordinate (B1) at (2,0);
\coordinate (B2) at (3,0.5);
\coordinate (B3) at (4,0);
\coordinate (B4) at (3,-.5);
\draw[thick] (B2)--(B4);
\foreach \x in {B2,B4}
\draw[fill, white] (\x) circle[radius=3mm];
\foreach \x/\y in {B1/1,B2/2,B3/3,B4/4}
\draw (\x) circle[radius=3mm] node{\y};
\end{scope}
\begin{scope}[xshift=9cm]
\draw (.5,0.15)node{$\xrightarrow{\sigma_2}$};
\coordinate (C1) at (2,0);
\coordinate (C2) at (3,-1);
\coordinate (C3) at (3,1);
\coordinate (C4) at (3,0);
\draw[thick] (C2)--(C3);
\foreach \x in {C2,C3,C4}
\draw[fill, white] (\x) circle[radius=3mm];
\foreach \x/\y in {C1/1,C2/2,C3/3,C4/4}
\draw (\x) circle[radius=3mm] node{\y};
\end{scope}
\begin{scope}[xshift=13cm]
\draw (.5,0.15)node{$\xrightarrow{\sigma_3}$};
\coordinate (D1) at (2,-.5);
\coordinate (D2) at (3,-.5);
\coordinate (D3) at (4,-.5);
\coordinate (D4) at (3,.5);
\draw[thick] (D2)--(D4)--(D1) (D3)--(D4);
\foreach \x in {D1,D2,D3,D4}
\draw[fill, white] (\x) circle[radius=3mm];
\foreach \x/\y in {D1/1,D2/2,D3/3,D4/4}
\draw (\x) circle[radius=3mm] node{\y};
\end{scope}
\end{tikzpicture}
\end{equation}
The Pr\"ufer code corresponding to a rooted labeled forest is given by $(p_1,\cdots,p_{n-1})$ where $p_1$ is the label of the unique node adjacent to the leaf of the forest with the largest label ($p_1=0$ if this maximal leaf is a root). When that leaf is removed from the forest, the label of the unique node adjacent to the leaf with the largest label is $p_2$ and so on. For example, in the third forest, the largest root is 2. Above 2 is $p_1=4$. After removing 2, the largest root is 4 and above 4 is $p_2=3$. What remains are two roots. So, $p_3=0$. The resulting Pr\"ufer code is $(4,3,0)$.

We now examine the same sequence of complete exceptional sequences and corresponding sequence of rooted labeled forests using our new direct bijection between these two structures. The parking function $(1,1,1,1)$ corresponds to the sequence of injective modules $(M_{14},M_{13},M_{12},M_{11})$. The sequence of braid moves $\sigma_1,\sigma_2,\sigma_3$ from above with corresponding rooted labeled forests, Pr\"ufer codes and parking functions are given by:
\[
\begin{array}{cccc}
\text{Exceptional sequence}  &  \text{Forest} & \text{Pr\"ufer code} & \text{Parking function} \\
\hline
  (M_{14},M_{13},M_{12},M_{11}) & %
\begin{tikzpicture}[scale=.55]
\draw[thick] (4,0)--(1,3);
\draw[white,fill] (4,0) circle[radius=4mm];
\draw (4,0) circle[radius=4mm];
\draw[white,fill] (3,1) circle[radius=4mm];
\draw (3,1) circle[radius=4mm];
\draw[white,fill] (2,2) circle[radius=4mm];
\draw (2,2) circle[radius=4mm];
\draw[white,fill] (1,3) circle[radius=4mm];
\draw (1,3) circle[radius=4mm];
\draw (4,0) node{4};
\draw (3,1) node{3};
\draw (2,2) node{2};
\draw (1,3) node{1};
\end{tikzpicture}
& (3,2,1) & (1,4,1,2)\\
\xrightarrow{\sigma_1} (M_{44},M_{14},M_{12},M_{11})  &  
\begin{tikzpicture}[scale=.55]
%
\draw[thick] (1,1)--(2,2)--(4,0);
\draw[white,fill] (1,1) circle[radius=4mm];
\draw (1,1) circle[radius=4mm];
\draw[white,fill] (2,2) circle[radius=4mm];
\draw (2,2) circle[radius=4mm];
\draw[white,fill] (3,1) circle[radius=4mm];
\draw (3,1) circle[radius=4mm];
\draw[white,fill] (4,0) circle[radius=4mm];
\draw (4,0) circle[radius=4mm];
\draw (4,0) node{4};
\draw (3,1) node{3};
\draw (2,2) node{2};
\draw (1,1) node{1};
\end{tikzpicture}
 & (3,2,2)  & (1,4,1,3) \\
\xrightarrow{\sigma_2}(M_{44},M_{34},M_{14},M_{11})  &  
\begin{tikzpicture}[scale=.55]
%
\draw[thick] (1,0)--(3,2)--(4,1);
\draw[white,fill] (1,0) circle[radius=4mm];
\draw (1,0) circle[radius=4mm];
\draw[white,fill] (2,1) circle[radius=4mm];
\draw (2,1) circle[radius=4mm];
\draw[white,fill] (3,2) circle[radius=4mm];
\draw (3,2) circle[radius=4mm];
\draw[white,fill] (4,1) circle[radius=4mm];
\draw (4,1) circle[radius=4mm];
\draw (1,0) node{1};
\draw (2,1) node{2};
\draw (3,2) node{3};
\draw (4,1) node{4};
\end{tikzpicture}
 & (3,2,3)  & (3,1,3,1) \\
\xrightarrow{\sigma_3}(M_{44},M_{34},M_{24},M_{14})  & 
\begin{tikzpicture}[scale=.55]
\draw[thick] (1,0)--(4,3);
\draw[white,fill] (1,0) circle[radius=4mm];
\draw (1,0) circle[radius=4mm];
\draw[white,fill] (2,1) circle[radius=4mm];
\draw (2,1) circle[radius=4mm];
\draw[white,fill] (3,2) circle[radius=4mm];
\draw (3,2) circle[radius=4mm];
\draw[white,fill] (4,3) circle[radius=4mm];
\draw (4,3) circle[radius=4mm];
\draw (1,0) node{1};
\draw (2,1) node{2};
\draw (3,2) node{3};
\draw (4,3) node{4};
\end{tikzpicture}
  & (2,3,4)  & (2,4,2,1)
\end{array}
\]

This example uses Figure \ref{fig: braid example} to perform the braid group action on our exceptional sequence. The action is easy to perform on the rooted labeled forests. We saw that the braid group action on parking functions from \cite{GG} was equivalent to a different and more complicated action of the braid group on these forests as shown in \eqref{eq: braid action on forests from GG action}.

\section{Action of the Garside element}\label{sec 4}\label{garfieldelement}

We recall the definition of the fundamental braid $\delta$ and recall the well-known action of $\delta$ on complete exceptional sequences over any hereditary algebra. The corresponding action of $\delta$ on forests has an easy description. We also describe the action of the Garside element $\Delta$ on rooted labeled forests. $\Delta$ is an important element of the braid group for many reasons \cite{Garside}. For example, $\Delta^2=\delta^n$ generates the center of the braid group \cite{Chow} and the action of $\Delta$ on complete exceptional sequences over any hereditary algebra converts support tilting objects to $c$-vectors. See section \ref{ss: clusters and exc seq} below for a full explanation. (This is a comment made without proof in \cite{BRT}.) In section \ref{ss: extended braid group} we use $\Delta$ to describe an action of the extended braid group $\widetilde B_n\rtimes \ZZ_2$ on the set of rooted labeled forests. 

To do these detailed computations, we need more notation about exceptional sequences. First, the \emph{Euler pairing} for $\Lambda$-modules $X,Y$, also called the \emph{Euler characteristic} of $\Lambda$ \cite{ASS} is the biadditive integer pairing on $\mathbb Z^n$ characterized by the property that
\[
	\left<
		\undim X,\undim Y
	\right>_\Lambda:= \dim \Hom_\Lambda(X,Y)-\dim \Ext_\Lambda(X,Y).
\]
We will use the shorthand notation:
\[
	\chi_\Lambda(X,Y):=\left<
		\undim X,\undim Y
	\right>_\Lambda.
\]

\subsection{The fundamental braid $\delta$}\label{ss: fundamental braid}

This is the braid group element
\[
	\delta=\delta_n=\sigma_1\sigma_2\cdots \sigma_{n-1}.
\]
This pushed the last ($n$th) strand underneath the first $n-1$ strands to put it on the far left. (See Figure \ref{Fig: delta}.) The action of $\delta$ on an exceptional sequence is
\[
	\delta(E_1,E_2,\cdots,E_n)=( E_n',E_1,E_2,\cdots,E_{n-1}).
\]
\begin{rem}\label{rem: action of delta on exc seq}
Each object in a complete exceptional sequence is uniquely determined up to isomorphism by its position and the other objects. In this case we must have
\begin{enumerate}
\item $E_n'=\tau E_n$, the Auslander-Reiten translation of $E_n$, if $E_n$ is not projective since, in that case,
\[
	\chi_\Lambda(E_n,X)=-\chi_\Lambda(X,\tau E_n).
\]
\item When $E_n=P_i$ is projective with simple top $S_i$ then
\[
	\chi_\Lambda(P_i,X)=\dim X_i=\chi_\Lambda(X,I_i)
\]
where $I_i$ is the injective envelope of $S_i$. So, $E_n'=I_i$ in this case.
\end{enumerate}
(2) can also be seen combinatorially using Proposition \ref{root prop} which implies that $v_n$ is a root of $F$ if and only if $v_1'$ is a root of $F'=\delta F$. But this is equivalent to $E_n'$, the first object of $\delta E_\ast$ being injective. We use the shorthand notation
\[
	\tau_\Lambda^\ast X:=\begin{cases} \tau X & \text{if $X$ is not projective}\\
    I_i & \text{if $X=P_i$}
    \end{cases}
\]
\end{rem}

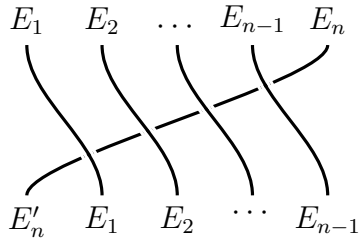
\begin{figure}[htbp]
\begin{center}
\begin{tikzpicture}
\draw[very thick] (4,2) .. controls (4,1.5) and (0,.5)..(0,0);
\begin{scope}[xshift=.5mm]
\draw[white,very thick] (1,0) .. controls (1,.7) and (0,1.3)..(0,2);
\draw[white,very thick] (2,0) .. controls (2,.7) and (1,1.3)..(1,2);
\draw[white,very thick] (3,0) .. controls (3,.7) and (2,1.3)..(2,2);
\draw[white,very thick] (4,0) .. controls (4,.7) and (3,1.3)..(3,2);
\end{scope}
\begin{scope}[xshift=1mm]
\draw[white,very thick] (1,0) .. controls (1,.7) and (0,1.3)..(0,2);
\draw[white,very thick] (2,0) .. controls (2,.7) and (1,1.3)..(1,2);
\draw[white,very thick] (3,0) .. controls (3,.7) and (2,1.3)..(2,2);
\draw[white,very thick] (4,0) .. controls (4,.7) and (3,1.3)..(3,2);
\end{scope}
\begin{scope}[xshift=-.5mm]
\draw[white,very thick] (1,0) .. controls (1,.7) and (0,1.3)..(0,2);
\draw[white,very thick] (2,0) .. controls (2,.7) and (1,1.3)..(1,2);
\draw[white,very thick] (3,0) .. controls (3,.7) and (2,1.3)..(2,2);
\draw[white,very thick] (4,0) .. controls (4,.7) and (3,1.3)..(3,2);
\end{scope}
\begin{scope}[xshift=-1mm]
\draw[white,very thick] (1,0) .. controls (1,.7) and (0,1.3)..(0,2);
\draw[white,very thick] (2,0) .. controls (2,.7) and (1,1.3)..(1,2);
\draw[white,very thick] (3,0) .. controls (3,.7) and (2,1.3)..(2,2);
\draw[white,very thick] (4,0) .. controls (4,.7) and (3,1.3)..(3,2);
\end{scope}
\begin{scope}
\draw[very thick] (1,0) .. controls (1,.7) and (0,1.3)..(0,2);
\draw[very thick] (2,0) .. controls (2,.7) and (1,1.3)..(1,2);
\draw[very thick] (3,0) .. controls (3,.7) and (2,1.3)..(2,2);
\draw[very thick] (4,0) .. controls (4,.7) and (3,1.3)..(3,2);
\draw (0,2) node[above]{$E_1$};
\draw (1,2) node[above]{$E_2$};
\draw (2,2) node[above]{$\cdots$};
\draw (3,2) node[above]{$E_{n-1}$};
\draw (4,2) node[above]{$E_{n}$};
\draw (0,0) node[below]{$E_n'$};
\draw (1,0) node[below]{$E_1$};
\draw (2,0) node[below]{$E_2$};
\draw (3,0) node[below]{$\cdots$};
\draw (4,0) node[below]{$E_{n-1}$};
\end{scope}
\end{tikzpicture}
\caption{The braid $\delta$ pulls the last strand to the first position underneath the other strands. On exceptional sequences, it moves the last object $E_n$ to the first position and changes it to $E_n'=\tau_\Lambda^\ast E_n$.}
\label{Fig: delta}
\end{center}
\end{figure}

Since the formula for $\delta$ depends on whether $E_n$ is projective, we need a combinatorial characterization of vertices corresponding to the projective (and injective) objects. 

\begin{lem}\label{lem: projective vertices}
The projective objects $P_1\subset P_2\subset\cdots\subset P_p$ of a complete exceptional sequence correspond to vertices $v_{j_1},v_{j_2},\cdots,v_{j_p}$ of the forest $F$ given as follows.
\begin{enumerate}
\item $v_{j_p}$ is the last root of $F$, i.e., the other roots have smaller labels.
\item $v_{j_i}$ is the last relatively projective child of $v_{j_{i+1}}$, i.e., $j_i$ is maximal among labels of children of $v_{j_{i+1}}$ so that $j_i<j_{i+1}$.
\end{enumerate}
\end{lem}

\begin{lem}\label{lem: injective vertices}
The injective objects $I_1\onto I_2\onto\cdots\onto I_q$ of a complete exceptional sequence correspond to vertices $v_{k_1},v_{k_2},\cdots,v_{k_q}$ of the forest $F$ given as follows.
\begin{enumerate}
\item $v_{k_1}$ is the first root of $F$, i.e., the other roots have larger labels.
\item $v_{k_{i+1}}$ is the first relatively injective child of $v_{k_{i}}$, i.e., $k_{i+1}$ is minimal among labels of children of $v_{k_{i}}$ with $k_{i+1}>k_{i}$.
\end{enumerate}
\end{lem}

We call the vertices $v_j$ satisfying Lemma \ref{lem: projective vertices} \emph{projective vertices} and those satisfying Lemma \ref{lem: injective vertices} the \emph{injective vertices} of the rooted labeled forest $F$. These lemmas imply the following.

\begin{lem}\label{lem: when v1, vn are roots}
In a complete exceptional sequence for linear $A_n$, $E_n$ is a projective module if and only if $v_n$ is a root of $F$. Similarly, $E_1$ is injective if and only if $v_1$ is a root.
\end{lem}

\begin{thm}\label{thm: action of delta on F}
Let $F, \delta F$ be the rooted labeled forests of $E_\ast=(E_1,\cdots,E_n)$ and $\delta E_\ast=(\tau^\ast E_n,E_1,\cdots,E_{n-1})$ with labeled vertices $v_i$ and $v_i'$ respectively.
\begin{enumerate}
\item If $v_n$ is not a root of $F$, then $\delta F$ is equal to $F$ with the labels cyclically permuted so that $v_1'=v_n$ and $v_i'=v_{i-1}$ for $1<i\le n$.
\item If $v_n$ is a root of $F$, then $\delta F$ is obtained by cyclically permuting the labels of $F$, then exchanging $v_1'=v_n$ with the master root $v_0$.
\end{enumerate} 
In both cases $\delta F_+\cong F_+$ as unlabeled trees.
\end{thm}

\begin{proof}
If $v_n$ is not a root of $F$ then, by Lemma \ref{lem: when v1, vn are roots}, $E_n$ is not projective. So, $E_n'=\tau E_n$ which has the same length as $E_n$. Since the other objects are the same as before, all objects have the same length as before. So, the support containment relations cannot change. So, $\delta F$ must be the same forest as $F$ with labels cyclically permuted.

If $v_n$ is a root of $F$, then $E_n$ is projective. Every other object in the sequence has support contained either in the support of $E_n$ or the support of the corresponding injective object $E_n'=\tau_\Lambda^\ast E_n$. Therefore, the new root $v_1'$ of $\delta F$ covers exactly those vertices of $F$ which were not descendants of $v_n$. This is equivalent to saying that $v_n$ has been exchanged with the master root and relabeled $v_1'$.
\end{proof}

\begin{figure}[ht]
\begin{center}
\begin{tikzpicture}
\begin{scope}[xshift=-2.5cm] 
\begin{scope}[xshift=2.5cm]
\draw[thick] (0.5,-.5)--(1,.5)--(1.5,-.5);
\draw[white,fill] (1,.5) circle[radius=3mm];
\draw (1,.5) circle[radius=3mm];
\draw[white,fill] (1.5,-.5) circle[radius=3mm];
\draw (1.5,-.5) circle[radius=3mm];
\draw[white,fill] (0.5,-.5) circle[radius=3mm];
\draw (0.5,-.5) circle[radius=3mm];
\draw (1,.5) node{2};
\draw (.5,-.5) node{1};
\draw (1.5,-.5) node{3};
\end{scope}
\begin{scope}[xshift=2cm]
\draw[thick,->] (3.1,0)--(3.9,0);
\draw (3.5,.3)node{$\delta$};
\end{scope}
\begin{scope}[xshift=6cm]
\draw[thick,->] (3.1,0)--(3.9,0);
\draw (3.5,.3)node{$\delta$};
\end{scope}
\begin{scope}[xshift=6.5cm]
\draw[thick] (0.5,-.5)--(1,.5)--(1.5,-.5);
\draw[white,fill] (1,.5) circle[radius=3mm];
\draw (1,.5) circle[radius=3mm];
\draw[white,fill] (1.5,-.5) circle[radius=3mm];
\draw (1.5,-.5) circle[radius=3mm];
\draw[white,fill] (0.5,-.5) circle[radius=3mm];
\draw (0.5,-.5) circle[radius=3mm];
\draw (1,.5) node{3};
\draw (.5,-.5) node{2};
\draw (1.5,-.5) node{1};
\end{scope}
\begin{scope}[xshift=11cm] 
\draw[white,fill] (0,0) circle[radius=3mm];
\draw (0,0) circle[radius=3mm];
\draw[white,fill] (1,0) circle[radius=3mm];
\draw (1,0) circle[radius=3mm];
\draw[white,fill] (2,0) circle[radius=3mm];
\draw (2,0) circle[radius=3mm];
\draw (0,0) node{1};
\draw (1,0) node{2};
\draw (2,0) node{3};
\end{scope}
\begin{scope}[xshift=15.5cm]
\draw[thick] (0.5,-.5)--(1,.5)--(1.5,-.5);
\draw[white,fill] (1,.5) circle[radius=3mm];
\draw (1,.5) circle[radius=3mm];
\draw[white,fill] (1.5,-.5) circle[radius=3mm];
\draw (1.5,-.5) circle[radius=3mm];
\draw[white,fill] (0.5,-.5) circle[radius=3mm];
\draw (0.5,-.5) circle[radius=3mm];
\draw (1,.5) node{1};
\draw (.5,-.5) node{2};
\draw (1.5,-.5) node{3};
\end{scope}
\begin{scope}[xshift=14cm]
\draw[thick,->] (.1,0)--(.9,0);
\draw (.5,.3)node{$\delta$};
\end{scope}
\end{scope} 
\end{tikzpicture}
\caption{Theorem \ref{thm: action of delta on F} in the case $n=3$ is illustrated (from Figure \ref{fig: A3 example 1}).}
\label{fig: A3 example delta}
\end{center}
\end{figure}
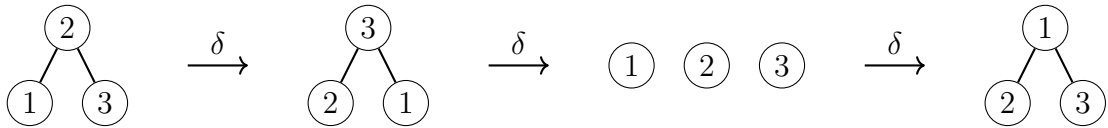

\begin{cor}\label{cor: action of central element Delta 2=delta to n}
Let $F,\delta^nF$ be the rooted labeled forests of $E_\ast=(E_1,\cdots,E_n)$ and $\delta^nE_\ast=(\tau^\ast E_1,\cdots,\tau^\ast E_n)$. Let $v_{j_1},\cdots,v_{j_p}$ be the projective vertices of $F$. Then $\delta^nF_+\cong F_+$ with projective vertices together with the master root $v_0$ cyclically permuted and all other vertices remaining the same. Thus $v_{j_p}'=v_0$, $v_0'=v_{j_1}$, $v_{j_i}'=v_{j_{i+1}}$ for $1\le i<p$ and $v_k'=v_k$ for all other values of $k$. In particular, $v_{j_1}',\cdots,v_{j_p}'$ are the injective vertices of $\delta^nF$ and the roots of $\delta^nF$ consist of $v_{j_1}'$ and the children of $v_{j_1}$ in $F$. (See Figure \ref{fig: Delta 2}.)
\end{cor}

\begin{proof}
The first part follows from Theorem \ref{thm: action of delta on F}. The last part follows from Lemma \ref{lem: injective vertices} since $j_1$ must be smaller than the labels of the other roots of $\delta^nF$. These other roots in $\delta^nF$ were the children of $v_{j_1}$ in $F$ which must have larger index by Lemma \ref{lem: projective vertices}.2.
\end{proof}

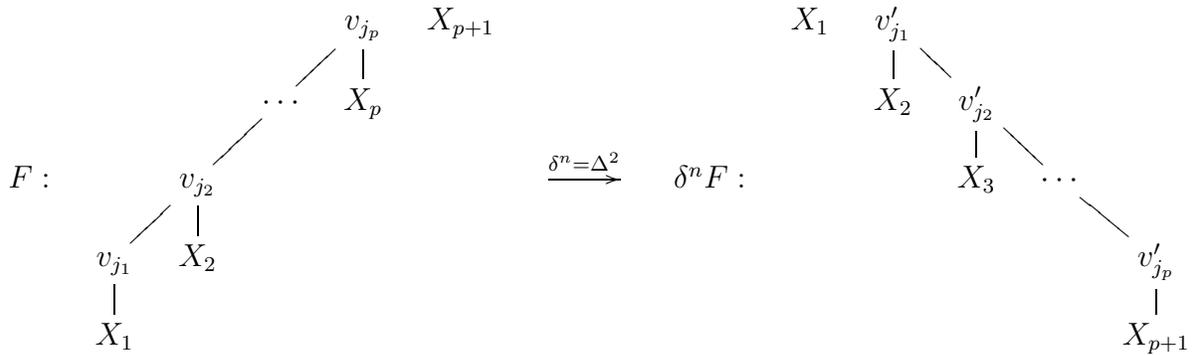
\begin{figure}[htbp]
\begin{center}
\[
\xymatrixrowsep{10pt}\xymatrixcolsep{10pt}
\xymatrix{
&&&&v_{j_p}\ar@{-}[d]\ar@{-}[ld] &X_{p+1}&&&& & X_1 & v_{j_1}'\ar@{-}[d]\ar@{-}[dr]\\
& &  & \cdots\ar@{-}[dl] & X_p &&& & &&  & X_2 & v_{j_2}'\ar@{-}[rd] & 	\\
F:& & v_{j_2}\ar@{-}[dl]\ar@{-}[d] &&  &&  \ar[rr]^{\delta^n=\Delta^2}&&&\delta^nF:  &  &  & X_3\ar@{-}[u] & \cdots\ar@{-}[dr] &	\\
&v_{j_1}\ar@{-}[d] & X_2 &  &&  &&&&&  &  &  &  &  v_{j_p}'\ar@{-}[d]	\\
&X_1 &  &  &  & && &  &&  && &  &X_{p+1}	\\
}
\]
\caption{Illustrating Corollary \ref{cor: action of central element Delta 2=delta to n}, the projective vertices of $F$ become the injective vertices of $\delta^nF$. The children of the smallest projective vertex $v_{j_1}$ in $F$ have larger indices and become the other roots of $\delta^nF$ with the injective vertex $v_{j_1}'$ being the first root of $\delta^nF=\Delta^2F$.}
\label{fig: Delta 2}
\end{center}
\end{figure}

\subsection{Garside element $\Delta$}\label{ss: Garside}

The \emph{Garside element} of the braid group $B_n$ is given by
\[
	\Delta=\delta_{n-1}\delta_{n-2}\cdots\delta_2\delta_1
\]
where, for each $k<n$, $\delta_k$ is given by
\[
	\delta_k=\sigma_1\sigma_2\cdots\sigma_k.
\]
In particular $\delta_{n-1}=\delta$ is the fundamental braid. The commutator relation:
\[
	\Delta \sigma_i \Delta^{-1}=\sigma_{n-i}
\]
implies that $\Delta^2$ lies in the center of the braid group $B_n$. In fact, $\Delta^2$ generates the center of $B_n$ by \cite{Chow}. It is also an easy observation that $\Delta^2=\delta^n$. So, the action of the central element $\Delta^2$ is given by Corollary \ref{cor: action of central element Delta 2=delta to n} and Figure \ref{fig: Delta 2}. 

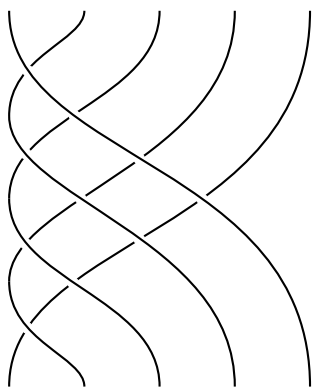
\begin{figure}[htbp]
\begin{center}
\begin{tikzpicture}
\coordinate (A1) at (0,5);
\coordinate (A2) at (0,3);
\coordinate (A3) at (4,3);
\coordinate (A4) at (4,0);
\coordinate (B1) at (0,3.6);
\coordinate (B2) at (0,2.5);
\coordinate (B3) at (3,2);
\coordinate (B4) at (3,0);
\coordinate (C1) at (0,2.5);
\coordinate (C2) at (0,1.4);
\coordinate (C3) at (2,1.2);
\coordinate (C4) at (2,0);
\coordinate (D1) at (0,1.4);
\coordinate (D2) at (0,0.6);
\coordinate (D3) at (1,0.4);
\coordinate (D4) at (1,0);
\begin{scope}[yscale=-1,yshift=-5cm]
\draw[thick] (0,5)..controls (0,3) and (4,3)..(4,0);
\draw[thick] (0,3.6)..controls (0,2.5) and (3,2)..(3,0);
\draw[thick] (0,2.5)..controls (0,1.4) and (2,1.2)..(2,0);
\draw[thick] (0,1.4)..controls (0,0.6) and (1,0.4)..(1,0);
\end{scope}
\begin{scope}
\draw[fill,white] (.24,.78) circle[radius=2pt];
\draw[thick] (D1)..controls (D2) and (D3)..(D4);
\draw[fill,white] (.85,1.38) circle[radius=2pt];
\draw[fill,white] (.22,1.9) circle[radius=2pt];
\draw[thick] (C1)..controls (C2) and (C3)..(C4);
\draw[fill,white] (.22,3.1) circle[radius=2pt];
\draw[fill,white] (.95,2.5) circle[radius=2pt];
\draw[fill,white] (1.73,1.96) circle[radius=2pt];
\draw[thick] (B1)..controls (B2) and (B3)..(B4);
\draw[fill,white] (.25,4.2) circle[radius=2pt];
\draw[fill,white] (.87,3.6) circle[radius=2pt];
\draw[fill,white] (1.73,3.03) circle[radius=2pt];
\draw[fill,white] (2.56,2.5) circle[radius=2pt];
\draw[thick] (A1)..controls (A2) and (A3)..(A4);
\end{scope}
\end{tikzpicture}
\caption{The Garside braid $\Delta$ for $n=5$. Each strand goes over the ones on its right and under the ones on its left.}
\label{fig: Garside drawing}
\end{center}
\end{figure}
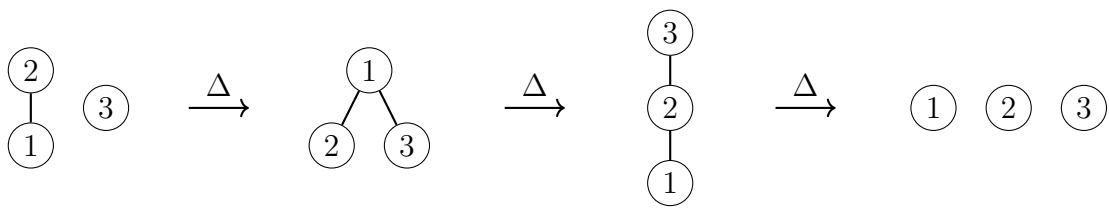
\begin{figure}[ht]
\begin{center}
\begin{tikzpicture}
\begin{scope}[yshift=2.5cm] 
\begin{scope}[xshift=-3.5cm]
\draw[thick,->] (2.1,0)--(2.9,0);
\draw (2.5,.3)node{$\Delta$};
\end{scope}
\begin{scope}[xshift=-4cm]
\draw[white,fill] (1.5,0) circle[radius=3mm];
\draw (1.5,0) circle[radius=3mm];
\draw[thick] (0.5,.5)--(0.5,-.5);
\draw[white,fill] (0.5,.5) circle[radius=3mm];
\draw (0.5,.5) circle[radius=3mm];
\draw[white,fill] (0.5,-.5) circle[radius=3mm];
\draw (0.5,-.5) circle[radius=3mm];
\draw (0.5,.5) node{2};
\draw (.5,-.5) node{1};
\draw (1.5,0) node{3};
\end{scope}
\begin{scope}[xshift=-3mm]
\draw[thick,->] (3.1,0)--(3.9,0);
\draw (3.5,.3)node{$\Delta$};
\end{scope}
\begin{scope}[xshift=4.3cm]
\draw[thick,->] (2.1,0)--(2.9,0);
\draw (2.5,.3)node{$\Delta$};
\end{scope}
\begin{scope}
\draw[thick] (0.5,-.5)--(1,.5)--(1.5,-.5);
\draw[white,fill] (1,.5) circle[radius=3mm];
\draw (1,.5) circle[radius=3mm];
\draw[white,fill] (1.5,-.5) circle[radius=3mm];
\draw (1.5,-.5) circle[radius=3mm];
\draw[white,fill] (0.5,-.5) circle[radius=3mm];
\draw (0.5,-.5) circle[radius=3mm];
\draw (1,.5) node{1};
\draw (.5,-.5) node{2};
\draw (1.5,-.5) node{3};
\end{scope}
\begin{scope}[xshift=5cm]
\draw[thick] (0,1)--(0,-1);
\draw[white,fill] (0,0) circle[radius=3mm];
\draw (0,0) circle[radius=3mm];
\draw[white,fill] (0,1) circle[radius=3mm];
\draw (0,1) circle[radius=3mm];
\draw[white,fill] (0,-1) circle[radius=3mm];
\draw (0,-1) circle[radius=3mm];
\draw (0,1) node{3};
\draw (0,0) node{2};
\draw (0,-1) node{1};
\end{scope}
\begin{scope}[xshift=8.5cm] 
\draw[white,fill] (0,0) circle[radius=3mm];
\draw (0,0) circle[radius=3mm];
\draw[white,fill] (1,0) circle[radius=3mm];
\draw (1,0) circle[radius=3mm];
\draw[white,fill] (2,0) circle[radius=3mm];
\draw (2,0) circle[radius=3mm];
\draw (0,0) node{1};
\draw (1,0) node{2};
\draw (2,0) node{3};
\end{scope}
\end{scope} 
\end{tikzpicture}
\caption{The action of $\Delta=\sigma_1\sigma_2\sigma_1=\sigma_2\sigma_1\sigma_2$ is shown in the case $n=3$ using Figure \ref{fig: A3 example 1}. Proposition \ref{prop: properties of Delta F} explains the combinatorics of $\Delta$.
}
\label{fig: A3 example Delta}
\end{center}
\end{figure}

The Garside braid is shown in Figure \ref{fig: Garside drawing}. Examples of the action of $\Delta$ on exceptional sequences for $A_3$ were computed in Figure \ref{fig: A3 example 1} and are summarized in Figure \ref{fig: A3 example Delta} below. The use of the Garside element in cluster theory is explained in the next subsection. Here we discuss the action of $\Delta$ on forests.

We use the notation
\[
	\Delta(E_1,\cdots,E_n)=(E_1',\cdots,E_n').
\]
The first thing to notice is that $E_n'=E_1$ since the first strand goes over the other strands in each step of the braid $\Delta$. Similarly, every strand goes over the ones on its right and under the ones on its left.

Let $\mathcal C_k=Z^\perp$, the right perpendicular category of $Z=E_{k+1}\oplus\cdots \oplus E_n$. This is the category of all $\Lambda$-modules $X$ so that $
	\chi_\Lambda(E_j,X)=0
$ for all $j>k$.

\begin{lem} The sequence $(E_1,\cdots,E_k)$ is a complete exceptional sequence for the category $\cC_k$. 
\end{lem}

\begin{proof}
The objects $E_1,\cdots,E_k$ lie in $\cC_k$ by definition of an exceptional sequence. To show it is complete, suppose we could add an object of $\cC_k$ to this sequence. Then, that object could also be added to the complete exceptional sequence $(E_1,\cdots,E_n)$ to the left of $E_{k+1}$ contradicting its maximality. 
\end{proof}

We recall our main theorem (Theorem \ref{theorem A2}) which says that $v_k$ is descending in $F$ (either a root or the child of some $v_j$ for $j>k$) if and only if $E_k$ is relatively projective which means $E_k$ is a projective object of the category $\cC_k$.

\begin{lem}\label{lem: Ek'}
The object $E_{n-k+1}'$ in $\Delta E_\ast=(E_1',\cdots,E_n')$ is the unique indecomposable object of $\cC_k$ having the property that $(E_{n-k+1}',E_1,\cdots,E_{k-1})$ is a complete exceptional sequence for $\cC_k$. Consequently,
\begin{enumerate}
\item If $E_k$ is not relatively projective, then $E_{n-k+1}'=\tau_kE_k$ where $\tau_k$ is Auslander-Reiten translation in $\cC_k$. In that case, $\chi_\Lambda(E_k,E_{n-k+1}')=-1$.
\item If $E_k$ is relatively projective, then $E_{n-k+1}'$ is the injective envelope in $\cC_k$ of the relatively simple $\cC_k$-top of $E_k$. In this case, $\chi_\Lambda(E_k,E_{n-k+1}')=+1$.
\end{enumerate}
In particular, $E_1'=\tau^\ast E_n$ (just as in $\delta E_\ast$).
\end{lem}

\begin{proof}
The braid move $\delta_{k-2}\cdots\delta_1$ on $E_\ast$ produces the exceptional sequence
\[
	(E_{n-k+2}',\cdots,E_n',E_k,E_{k+1},\cdots,E_n)
\]
which implies that $(E_{n-k+2}',\cdots,E_n')$ is a complete exceptional sequence for $\cC_{k-1}$. Applying the braid move $\delta_{k-1}$ produces the exceptional sequence
\[
	(E_{n-k+1}',E_{n-k+2}',\cdots,E_n',E_{k+1},\cdots,E_n)
\]
which implies that $E_{n-k+1}'$ is the unique object of $\cC_k$ in the right perpendicular category of $\cC_{k-1}$. Equivalently, $(E_{n-k+1}',E_1,\cdots,E_{k-1})$ is a complete exceptional sequence for $\cC_k$ as claimed. The rest follows by analogy with Theorem \ref{thm: action of delta on F}.
\end{proof}

This lemma immediately implies the following.

\begin{lem}\label{lem: rel proj and rel inj in Delta E}
Let $\Delta E_\ast=(E_1',\cdots,E_n')$. Then $E_k$ is relatively projective in $E_\ast$ if and only if $E_{n-k+1}'$ is relatively injective in $\Delta E_\ast$.
\end{lem}

In the proof of Lemma \ref{lem: Ek'} we saw that
\[
	\delta_{k-1}^{-1}(E_{n-k+1}',E_{n-k+2}',\cdots,E_n')=(E_{n-k+2}',\cdots,E_n',E_k).
\]
If we apply this braid move to the last $k$ objects in $\Delta E_\ast$ we obtain the following.

\begin{lem}\label{lem: Ek is right perpendicular to everything}
With the notation $\Delta (E_1,\cdots,E_n) =(E_1',\cdots,E_n')$ we have that
\[
	(E_1',\cdots,E_{n-k}',E_{n-k+2}',\cdots,E_n',E_k)
\]
is a complete exceptional sequence. In particular,
\[
	\chi_\Lambda(E_k,E'_j)=0
\]
for $j\neq n-k+1$.
\end{lem}

\begin{prop}\label{prop: properties of Delta F}
Let $F'=\Delta F$ for a rooted labeled forest $F$ with $n$ vertices. Then
\begin{enumerate}
\item  $v_{j_1},\cdots, v_{j_p}$ are the projective vertices of $F$ if and only if $v_{n-j_1+1}',\cdots, v_{n-j_p+1}'$ are the roots of $\Delta F$.
\item $v_{k_1},\cdots, v_{k_r}$ are the roots of $F$ if and only if $v_{n-k_1+1}',\cdots, v_{n-k_r+1}'$ are the injective vertices of $\Delta F$.
\end{enumerate}
\end{prop}

\begin{proof}
Since $\Delta^2=\delta^n$, it follows from Corollary \ref{cor: action of central element Delta 2=delta to n} that (1) and (2) are equivalent. So, we prove only (1). 

Statement (1) says that $v_k$ is a projective vertex of $F$ if and only if $v_{n-k+1}'$ is a root of $F'=\Delta F$. We prove this by induction on $n-k$. For $k=n$, $E_n$ is projective if and only if $v_n$ is a root of $F$ if and only if the first vertex $v_1'$ of $\Delta F$ is a root (See Remark \ref{rem: action of delta on exc seq}). So, the statement holds for $k=n$.

Now suppose the statement holds for $k+1$. Then, we observe that the braid move $\sigma_k$ moves $E_k$ to the $k+1$st position in $\sigma_k E_\ast$:
\[
	\sigma_k(E_1,\cdots,E_k,\cdots,E_n)=(E_1,\cdots,E_{k-1},E_{k+1}^\ast,E_k,E_{k+2},\cdots,E_n).
\]
Therefore, by induction on $n-k$ we have that $E_k$ is a projective module if and only if the $(n-k)$th vertex of $\Delta \sigma_kF=\sigma_{n-k}\Delta F$ is a root. By Proposition \ref{root prop}, this is equivalent to the $(n-k+1)$st vertex of $\Delta F$ being a root. Thus the statement holds for $k$.

This proves (1) in the proposition.
\end{proof}

\begin{cor}
$\Delta E_\ast$ has a simple injective object if and only if $E_\ast$ has only one injective object. Similarly, $E_\ast$ has a simple projective object if and only if $\Delta E_\ast$ has only one projective object.
\end{cor}

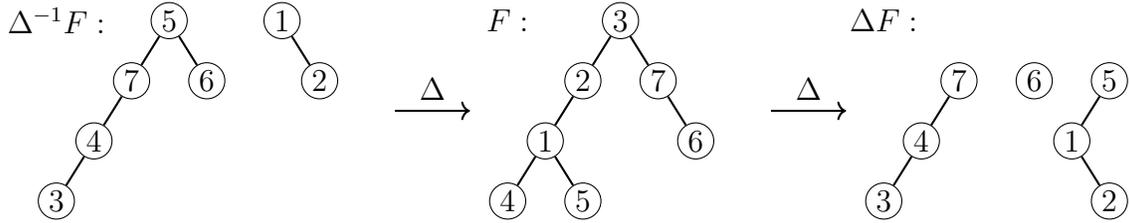
\begin{figure}[htbp]
\begin{center}
\begin{tikzpicture}[yscale=.8]
\begin{scope}
\draw[thick] (0,0)--(.5,1)--(1,0) (.5,1)--(1,2)--(1.5,3)--(2,2)--(2.5,1);
\draw[white,fill] (0,0) ellipse[x radius=2.4mm, y radius=3mm];
\draw (0,0) circle[x radius=2.4mm, y radius=3mm] (0,0)node{4};
\draw (0,3) node{$F:$};
\draw[white,fill] (1,0) circle[x radius=2.4mm, y radius=3mm];
\draw (1,0) circle[x radius=2.4mm, y radius=3mm] (1,0)node{5};
\draw[white,fill] (0.5,1) circle[x radius=2.4mm, y radius=3mm];
\draw (0.5,1) circle[x radius=2.4mm, y radius=3mm] (0.5,1)node{1};
\draw[white,fill] (2.5,1) circle[x radius=2.4mm, y radius=3mm];
\draw (2.5,1) circle[x radius=2.4mm, y radius=3mm] (2.5,1)node{6};
\draw[white,fill] (1.5,3) circle[x radius=2.4mm, y radius=3mm];
\draw (1.5,3) circle[x radius=2.4mm, y radius=3mm] (1.5,3)node{3};
\draw[white,fill] (2,2) circle[x radius=2.4mm, y radius=3mm];
\draw (2,2) circle[x radius=2.4mm, y radius=3mm] (2,2)node{7};
\draw[white,fill] (1,2) circle[x radius=2.4mm, y radius=3mm];
\draw (1,2) circle[x radius=2.4mm, y radius=3mm] (1,2)node{2};
\end{scope}
\begin{scope}[xshift=-4.5cm]
\draw[thick,->] (3,1.5)--(4,1.5);
\draw (3.5,1.5) node[above]{$\Delta$};
\end{scope}
\begin{scope}[xshift=-6cm]
\draw (0,3) node{$\Delta^{-1}F:$};
\draw[thick] (0,0)--(.5,1) (.5,1)--(1,2)--(1.5,3)--(2,2) (3,3)--(3.5,2);
\draw[white,fill] (0,0) circle[x radius=2.4mm, y radius=3mm];
\draw (0,0) circle[x radius=2.4mm, y radius=3mm] (0,0)node{3};
\draw[white,fill] (0.5,1) circle[x radius=2.4mm, y radius=3mm];
\draw (0.5,1) circle[x radius=2.4mm, y radius=3mm] (0.5,1)node{4};
\draw[white,fill] (3,3) circle[x radius=2.4mm, y radius=3mm];
\draw (3,3) circle[x radius=2.4mm, y radius=3mm] (3,3)node{1};
\draw[white,fill] (3.5,2) circle[x radius=2.4mm, y radius=3mm];
\draw (3.5,2) circle[x radius=2.4mm, y radius=3mm] (3.5,2)node{2};
\draw[white,fill] (1.5,3) circle[x radius=2.4mm, y radius=3mm];
\draw (1.5,3) circle[x radius=2.4mm, y radius=3mm] (1.5,3)node{5};
\draw[white,fill] (2,2) circle[x radius=2.4mm, y radius=3mm];
\draw (2,2) circle[x radius=2.4mm, y radius=3mm] (2,2)node{6};
\draw[white,fill] (1,2) circle[x radius=2.4mm, y radius=3mm];
\draw (1,2) circle[x radius=2.4mm, y radius=3mm] (1,2)node{7};
\end{scope}
\begin{scope}[xshift=5mm]
\draw[thick,->] (3,1.5)--(4,1.5);
\draw (3.5,1.5) node[above]{$\Delta$};
\end{scope}
\begin{scope}[xshift=5cm]
\draw (0,3) node{$\Delta F:$};
\draw[thick] (0,0)--(.5,1)--(1,2) (3,0)--(2.5,1)--(3,2);
\draw[white,fill] (1,2) circle[x radius=2.4mm, y radius=3mm];
\draw (1,2) circle[x radius=2.4mm, y radius=3mm] (1,2)node{7};
\draw[white,fill] (0,0) circle[x radius=2.4mm, y radius=3mm];
\draw (0,0) circle[x radius=2.4mm, y radius=3mm] (0,0)node{3};
\draw[white,fill] (.5,1) circle[x radius=2.4mm, y radius=3mm];
\draw (.5,1) circle[x radius=2.4mm, y radius=3mm] (.5,1)node{4};
\draw[white,fill] (2,2) circle[x radius=2.4mm, y radius=3mm];
\draw (2,2) circle[x radius=2.4mm, y radius=3mm] (2,2)node{6};
\draw[white,fill] (3,2) circle[x radius=2.4mm, y radius=3mm];
\draw (3,2) circle[x radius=2.4mm, y radius=3mm] (3,2)node{5};
\draw[white,fill] (2.5,1) circle[x radius=2.4mm, y radius=3mm];
\draw (2.5,1) circle[x radius=2.4mm, y radius=3mm] (2.5,1)node{1};
\draw[white,fill] (3,0) circle[x radius=2.4mm, y radius=3mm];
\draw (3,0) circle[x radius=2.4mm, y radius=3mm] (3,0)node{2};
\end{scope}
\end{tikzpicture}
\caption{Illustrating Proposition \ref{prop: properties of Delta F}: The projective vertices $v_1,v_2,v_3$ of $F$ become the roots $v_7,v_6,v_5$ resp. of $\Delta F$. The only root $v_3$ of $F$ becomes the only injective vertex $v_5$ of $\Delta F$. The injective vertices $v_3,v_7$ of $F$ correspond to the last injective $v_5$ in $\Delta F$ and its child $v_1$. The first projective in $F$ and its children become the projective vertices of $\Delta F$. For the placement of $v_2$ in $\Delta F$ we go to $\Delta^{-1}F$ then apply $\delta^7=\Delta^2$.}
\label{fig: example of Delta on A7}
\end{center}
\end{figure}

\subsection{Clusters and signed exceptional sequences}\label{ss: clusters and exc seq}

The Garside element $\Delta$ takes a ``support tilting set'' for any hereditary algebra to the corresponding ``signed exceptional sequence'' \cite{IT13}. This is a comment made in \cite{BRT} using different terminology (before the introduction of signed exceptional sequences in \cite{IT13}).

\begin{defn}
We define a \emph{support tilting set} for a hereditary algebra $\Lambda$ with $n$ simple objects to be a set of $n$ nonisomorphic modules and ``shifted projective modules'' $T_1,\cdots,T_n$ (each $T_i$ is either an indecomposable module $M_i$ or a shifted indecomposable projective module $P_k[1]$) having the property that $\Hom_{\cD^b}(T_i,T_j[1])=0$ in the bounded derived category $\cD^b={\cD^b(mod\text-\Lambda)}$ for all $i,j$. In other words, $\Ext_\Lambda(M_i,M_j)=0$ for all $i,j$ and $\Hom_\Lambda(P_k,M_i)=0$ for all $k,i$.

For example, in $A_3$, $\{S_1,S_3,P_2[1]\}$ is a support tilting set since $S_1,S_3$ do not extend each other ($\Ext_\Lambda(S_1,S_3)=\Ext_\Lambda(S_3,S_1)=0$) and $\Hom_\Lambda(P_2,S_1\oplus S_3)=0$.
\end{defn}

\begin{defn}
By a \emph{signed exceptional sequence} we mean an exceptional sequence $E_\ast=(E_1,\cdots,E_n)$ together with a sequence $\varepsilon_\ast$ of signs $\varepsilon_i\in\{+1,-1\}$ so that $\varepsilon_i=-1$ implies $E_i$ is relatively projective. We write the sequence as $(X_1,\cdots,X_n)$ where
\[
	X_i=\begin{cases} E_i[1] & \text{if }\varepsilon_i=-1\\
    E_i& \text{otherwise}.
    \end{cases}
\]
For example, in $A_3$, $(S_1[1],S_3,P_2)$ is a signed exceptional sequence since $S_1$ is relatively projective. (The first object in an exceptional sequence is always relatively projective.)
\end{defn}

In \cite{IT13}, an explicit correspondence between signed exceptional sequences and support tilting objects is shown.

\begin{thm}\label{thm: IT13: signed exc seq}\cite{IT13}
Over any hereditary algebra, there is a 1-1 correspondence between signed exceptional sequences and ordered support tilting sets.
\end{thm}

Thus, for type $A_n$, the number of signed exceptional sequences is $n!$ times the Catalan number $C_{n+1}=\frac1{n+2}\binom{2n+2}{n+1}$. For linear $A_n$, this can also be seen from our generating function $p_n(a,b,c)$ in Theorem \ref{thm: main theorem}. Since the relative projective objects can have any sign, the number of signed exceptional sequences is
\[
	p_n(2,1,2)=2(n-1+2(2))(n-2+2(4))\cdots(1+2(n))=2\frac{(2n+1)!}{(n+2)!}=n!C_{n+1}.
\]

For linear $A_n$, there is an explicit correspondence using the Garside element $\Delta$.

\begin{prop}\label{thm: clusters and c-vectors}
Any support tilting set can be arranged into a signed exceptional sequence $E_\ast=(E_1,\cdots,E_n)$ with signs $\varepsilon_1,\cdots,\varepsilon_n$. Furthermore, 
\begin{equation}\label{eq: Delta E}
	\Delta E_\ast=(E'_n,\cdots,E'_1)\text{ with signs } \varepsilon_n',\cdots,\varepsilon_1'
\end{equation}
is a signed exceptional sequence with the property that
\begin{equation}\label{eq: c-vectors}
	\epsilon_k\epsilon_j'\chi_\Lambda(E_k,E_j')=\delta_{kj}.
\end{equation}
In addition, the signed exceptional sequence \eqref{eq: Delta E} corresponds to the ordered support tilting set $(\varepsilon_nE_n,\cdots,\varepsilon_1E_1)$, which is $E_\ast$ with signs $\varepsilon_i$ in reverse order.
\end{prop}

\begin{rem}It is stated in \cite{BRT} that this holds for any hereditary algebra. Here we prove it in the case of linear $A_n$ using rooted labeled forests. Also, Equation \eqref{eq: c-vectors} implies that $-\varepsilon_j'\undim E_j'$ are the ``$c$-vectors'' of the ``cluster'' $\{\varepsilon_k\undim E_k\}$ if we take the ``initial seed'' to be the shifted projectives \cite{IOTW}.
\end{rem}

\begin{proof} It is an observation originally due to Schofield \cite{Schofield} that a support tilting set can be arranged in an exceptional sequence. In the case of finite type, this is very easy to see: take the left-to-right order of the objects as they appear in the Auslander-Reiten quiver. The shifted projective objects can be placed after these in increasing order of size. Since all negative (shifted) objects are projective, the result is a signed exceptional sequence $E_\ast,\varepsilon_\ast$.

To see that the exceptional sequence $\Delta E_\ast$ with signs $\varepsilon_\ast$ is a signed exceptional sequence we examine which objects have negative signs.
\begin{enumerate}
\item If $E_k$ is a projective module then $E_k$ and $E'_k$ have the same sign: $\varepsilon_k=\varepsilon'_k$. Since $v_k$ is a root of the forest $\Delta F$, $E'_k$ is relatively projective and either sign is allowed.
\item If $E_k$ is any relative projective object which is not projective then it is positive. So, $\varepsilon'_k=\varepsilon_k=+1$.
\item In all other cases, $E_k$ is relatively injective but not relatively projective. So, $E_k$ is positive and $\varepsilon'_k=-\varepsilon_k=-1$. This is OK since $E'_k$ is relatively projective (but not relatively injective) in $\Delta E_\ast$.
\end{enumerate}
The conclusion is that all objects in $\Delta E_\ast$ which are relatively projective but not relatively injective have a negative sign and those which are relatively injective but not relatively projective have positive sign. The roots might have either sign. So, $\Delta E_\ast$ with signs $\varepsilon_\ast'$ is a signed exceptional sequence.

By Lemma \ref{lem: Ek'} we have $\chi_\Lambda(E_k,E'_k)=\varepsilon_k\varepsilon'_k$. By Lemma \ref{lem: Ek is right perpendicular to everything}, $\chi_\Lambda(E_k,E'_j)=0$ when $j\neq k$. Equation \eqref{eq: c-vectors} follows.

To prove the last statement we recall some results from \cite{IT13}. By \cite[Corollary 2.17]{IT13}, the dimension vectors $\varepsilon_i'\undim E_i'$ of the signed exceptional sequence $(E_n',\cdots,E_1')$, with signs $\varepsilon_i'$, corresponds to the ordered support tilting set $(T_n,\cdots,T_1)$, with signs $\varepsilon_i$, are equal to the negative $c$-vector: $-\beta_i$ corresponding to the ordered support tilting set:
\begin{equation}\label{eq: -bi are dim Ei}
	\varepsilon_i'\undim E_i'=-\beta_i
\end{equation}
assuming that $(T_1,\cdots,T_n)$, which are $T_i$ in reverse order, form an exceptional sequence. This is the case here by construction if we take $T_i=E_{i}$.

By \cite[Theorem 2.14]{IT13}, the $c$-vectors $\beta_i$ corresponding to the support tilting set $\{\varepsilon_iT_i\}$ are the unique solutions of the equation:
\begin{equation}\label{eq: betaj are negative duals of ekTk}
    \epsilon_k\chi_\Lambda(T_k,\beta_j)=-\delta_{kj}.
\end{equation}
Combining \eqref{eq: betaj are negative duals of ekTk},\eqref{eq: -bi are dim Ei} and \eqref{eq: c-vectors} we see that the signed exceptional sequence \eqref{eq: Delta E} corresponds to the ordered cluster tilting set $(T_n,\cdots,T_1)$, with signs $\varepsilon_i$ as claimed.
\end{proof}

\begin{eg}\label{eg: Delta A2 cluster to exc seq}
For $A_2$, there are five support tilting objects. These are listed below, ordered as signed exceptional sequences. Applying $\Delta$ we get $\Delta(E_1,E_2)=(E_2',E_1')$ with $E_1'=E_1$ and $E_2'=-\tau E_2$ so that $\chi(E_i,E_j')=\varepsilon_i\varepsilon_j'\delta_{ij}$. We interpret $X[1]$ as negative $X$.
\begin{enumerate}
\item $(P_2,P_1)\xrightarrow\Delta(S_1,P_2)$
\item $(P_1,S_1)\xrightarrow\Delta(P_2[1],P_1)$
\item $(S_1,P_2[1])\xrightarrow\Delta(P_1[1],S_1)$
\item $(P_2[1],P_1[1])\xrightarrow\Delta(S_1[1],P_2[1])$
\item $(P_2,P_1[1])\xrightarrow\Delta(S_1[1],P_2)$
\end{enumerate} 
For example, in (2), $\tau S_1 =P_2$. So, $\Delta(P_1,S_1)=(-\tau S_1,P_1)=(P_2[1],P_1)$ with $\chi(S_1,P_2)=\dim\Hom_\Lambda(S_1,P_2)-\dim\Ext_\Lambda(S_1,P_2)=-1$, the sign of $P_2[1]$.
\end{eg}

\begin{eg}
None of the forests in Figure \ref{fig: example of Delta on A7} are examples of support tilting sets. In the first and second forest, there is a vertex with two relatively injective children. These extend each other. The last $\Delta F$ has three roots. The smaller two extend each other. (We could shift the projective module $E_7$.) In Figure \ref{fig: A3 example Delta}, the second and last forests are not support tilting for the same reasons. The second has two relatively injective children of the root and the last has three roots which is not allowed. The first forest is support tilting if we shift the projective module $E_3$. This gives $(S_2,I_2,S_3[1])$. The corresponding signed exceptional sequence is given by the second forest with its root shifted: $(P_1[1],S_1,S_2)$. To illustrate Equation \eqref{eq: c-vectors} note that $\Hom(I_2,S_1)=K$, but $\Hom(I_2,S_2)=0=\Ext(I_2,S_2)$ and $\Hom(I_2,P_1)=0=\Ext(I_2,P_1)$.
\end{eg}


\subsection{Action of extended braid group}\label{ss: extended braid group}

The \emph{extended braid group} is
\[
	\widetilde B_n:=B_n\rtimes_C \ZZ/2
\]
using the outer involution $C$ given by $\sigma_i \mapsto\sigma_i^{-1}$. We refer to $C$ as \emph{complex conjugation}. Equivalently, we use $D=\Delta C$ in $\widetilde B_n$ conjugation by which gives another involution:
\[
	D\sigma_iD= \sigma_{n-i}^{-1}.
\]
We call this \emph{duality} since it corresponds to $D:mod$-$\Lambda\to mod$-$\Lambda^{op}$ given by
\[
	DM=\Hom_K(M,K).
\]
This reverses the order of an exceptional sequence:
\[
	D(E_1,\cdots,E_n)=(DE_n,\cdots,DE_1)
\]
since $\chi_\Lambda(X,Y)=\chi_{\Lambda^{op}}(DY,DX)$. The operation $D$ gives an involution on the set of complete exceptional sequences of the linear $A_n$ quiver since this quiver is isomorphic to its opposite quiver. The modules $DE_i$ have the same support inclusion relations as the $E_i$. Therefore, duality $D$ acts on the corresponding rooted labeled forests by reversing the order of vertices: $DF\cong F$ with vertex labels $v_i'=v_{n-i+1}$.

Complex conjugation $C=D\Delta$ has a cleaner description than the Garside element:

\begin{prop}\label{prop: properties of C F}
Let $F'=C F=D\Delta F$. Then $F=D\Delta F'$, i.e., $C=D\Delta$ is an involution. Furthermore:
\begin{enumerate}
\item $v_{j_1},\cdots, v_{j_p}$ are the projective vertices of $F$ if and only if $v_{j_1}',\cdots, v_{j_p}'$ are the roots of $F'$.
\item $v_{k_1},\cdots,v_{k_q}$ are the injective vertices of $F$ with $k_1<k_2<\cdots<k_q$ if and only if $v_{k_1}'$ is the first projective vertex of $F'$ and $v_{k_2}',\cdots,v_{k_q}'$ are its children.
\item Dually, $v_{\ell_q}$ is the last injective vertex of $F$ and $v_{\ell_1},\cdots,v_{\ell_{q-1}}$ are its children if and only if $v_{\ell_1}'',\cdots,v_{\ell_q}''$ are the projective vertices of $F''=D\Delta^{-1}F$.
\end{enumerate}
\end{prop}

\subsection{Comments}
We note that the results of this paper partially answer a question in \cite{Sen} raised by the second author. The question was: ``Is it possible to interpret complete exceptional sequences of linear Nakayama algebras of rank $n$ by certain rooted labeled forests with $n$ labels?'' But another question is raised: In \cite{Sen} a bijection is constructed
between rooted labeled trees of height at most 1 and complete exceptional sequences for $A_n$ with radical square zero. How is this related to the bijection constructed in the
present paper? Can this bijection be extended to rooted labeled forests of larger bounded heights and exceptional sequences of other algebras?

\section{Acknowledgements}
The authors are grateful to Gordana Todorov for all her help and support. We also thank Olivier Bernardi for the speculation that the generating function for rooted labeled forests may be related to generating functions for exceptional sequences of type $A_n$. We thank Shujian Chen for discussions about chord diagrams and exceptional sequences. Also, Hugh Thomas explained the bijection from \cite{GY} to the first author many years ago. We would also like to thank the anonymous referee for many helpful comments and suggestions to improve the text.

The main results of this paper were presented as conjectures at the 2021 Workshop on Cluster Algebras and Related Topics held at the Morningside Center of Mathematics of the Chinese Academy of Sciences in Beijing held in August, 2021. The first author would like to thank the organizers of this conference for the opportunity to present this material. The first author also gratefully acknowledges the support of the Simons Foundation throughout this collaborative effort.

\end{document}